\documentclass[10pt,reqno]{amsart}
\usepackage[a4paper,top=3cm, bottom=3cm,left=3cm, right=3cm,twoside]{geometry}

\usepackage{amsthm}
 \usepackage{amsmath}
 \usepackage{amsfonts}
 \usepackage{amssymb}
 \usepackage{amscd}
 \usepackage{url}
 \usepackage[T1]{fontenc}
 \usepackage[utf8x]{inputenc}
 \usepackage[english]{babel}
 \usepackage[mathscr]{eucal}
 \usepackage{lmodern}
 \usepackage{bera}
 \usepackage{float} 
 \usepackage{bold-extra}
 \usepackage{textcomp}
 \usepackage{graphicx}
 \usepackage{caption}
\usepackage{subcaption} 
 \usepackage[dvipsnames,table]{xcolor}
 \usepackage{mathtools} 
 \usepackage{enumitem}
 \usepackage{tikz} 
 \usepackage{tikz-cd}
 \usepackage{wasysym}
 \usepackage{todonotes}
 \usepackage{fancybox}
 \usepackage{fancyhdr}
 \usepackage[nodisplayskipstretch]{setspace}
\usepackage{hyperref}

 \theoremstyle{definition}  
  \newtheorem{defn}{Definition}[section]
  \newtheorem{eg}[defn]{Example}

   \newtheorem{rmk}[defn]{Remark}

  \theoremstyle{plain}  
  \newtheorem{thm}[defn]{Theorem}
  \newtheorem{lem}[defn]{Lemma}
  \newtheorem{prop}[defn]{Proposition}
  \newtheorem{cor}[defn]{Corollary}

  \theoremstyle{remark}

 \renewcommand{\it}[1]{\textit{#1}}
 
 \renewcommand{\sf}[1]{\textsf{#1}}

 \newcommand{\mbb}[1]{\mathbb{#1}}
 \newcommand{\mcl}[1]{\mathcal{#1}}

 \newcommand{\msc}[1]{\mathscr{#1}}

 \newcommand{\ol}[1]{\overline{#1}}
 
 \newcommand{\wtilde}[1]{\widetilde{#1}}

 \newcommand{\abs}[1]{\left\lvert#1\right\rvert}

 \newcommand{\norm}[1]{\left\lVert#1\right\rVert}
 \newcommand{\bnorm}[1]{\bigl\lVert#1\bigr\rVert}

 \newcommand{\M}[1]{\mbb{M}_{#1}}

 \newcommand{\B}[1]{\msc{B}({#1})}    
 
 \newcommand{\TC}[1]{\msc{T}({#1})}
 \newcommand{\ST}[1]{\mcl{S}({#1})}
 \newcommand{\ranko}[2]{|{#1}\rangle\langle{#2}|}

 \newcommand{\ip}[1]{\langle#1\rangle}
 \newcommand{\bip}[1]{\bigl\langle#1\bigr\rangle}
 \newcommand{\Bip}[1]{\Bigl\langle#1\Bigr\rangle}
 \newcommand{\ran}[1]{\sf{range}(#1)}

 \newcommand{\mscriptsize}[1]{{\setlength{\arraycolsep}{.3ex}\text{\scriptsize$#1$}}}

 \newcommand{\Matrix}[1]{\begin{bmatrix}#1\end{bmatrix}}
 \newcommand{\sMatrix}[1]{\mscriptsize{\Matrix{#1}}}

 \DeclareMathOperator{\T}{\sf{T}}
 \DeclareMathOperator{\tr}{\sf{tr}}
 
 \DeclareMathOperator{\id}{\sf{id}}

 \DeclareMathOperator{\lspan}{\sf{span}}
 \DeclareMathOperator{\cspan}{\ol{\lspan}}

 \DeclareMathOperator*{\ssum}{\overline{\sum}}

 \numberwithin{equation}{section}
 \allowdisplaybreaks[4] 
 \setlist[enumerate]{font=\upshape,noitemsep, topsep=0pt} 
 \setlist[itemize]{noitemsep, topsep=0pt}

\makeatletter
\@namedef{subjclassname@2020}{\textup{2020} Mathematics Subject Classification}
\makeatother

\title{Degradable Strongly Entanglement Breaking Maps}

\author{Repana Devendra}
\address{Indian Institute of Technology Madras, Department of Mathematics, Chennai, Tamilnadu 600036, India}
\email{r.deva1992@gmail.com, ma16d020@smail.iitm.ac.in}

\author{Gunjan sapra}
\address{Indian Institute of Technology Madras, Department of Mathematics, Chennai, Tamilnadu 600036, India}
\email{gunjan.sapra27@gmail.com}

\author{K. Sumesh}
\address{Indian Institute of Technology Madras, Department of Mathematics, Chennai, Tamilnadu 600036, India}
\email{sumeshkpl@gmail.com, sumeshkpl@iitm.ac.in}

\begin{document}

\date{\today}

\begin{abstract}
 In this paper, we provide a structure theorem and various characterizations of degradable strongly entanglement breaking maps on separable Hilbert spaces. In the finite dimensional case,  we prove that unital degradable  entanglement breaking maps are precisely the $C^*$-extreme points of the convex set of unital entanglement breaking maps on  matrix algebras. Consequently, we get a structure for unital degradable positive partial transpose (PPT-) maps.
\end{abstract}



\maketitle


\section{Introduction}
  
 In \cite{DeSh05}, Devtak and Shor introduced the notion of degradable channels  via complementary channels.  Holevo (\cite{Hol07}) used complementary channels to  develop a class of  channels for which the multiplicativity/additivity problem (c.f.\cite{NiCh11}) has a positive solution.  Cubitt et al. (\cite{CRS08}) characterized all degradable qubit channels. In  \cite{WoPe07}, M.M Wolf et  al. introduced the notion of twisted diagonal channels and gave a simple criterion to check the degradability. A subclass of degradable channels, called  self-complementary channels, was studied in \cite{SRZ16}. It is known  that entanglement breaking maps (on matrix algebras) are anti-degradable maps but the converse is not true in general (c.f. \cite{DeSh05}, \cite{CRS08}). In this article, we investigate the mathematical structure of degradable (strongly) entanglement breaking maps and their connections with the notion of $C^*$-extreme points of unital entanglement breaking maps (\cite{BDMS23}) on matrix algebras. 
   
 This article is organized as follows. Section 2 contains basic definitions and theorems helpful in understanding the main results.  We give elaborated details about complementary and degradable maps as we could not find some of the mathematical details from the literature. In section 3, we present various characterizations (Theorem \ref{thm-DEB-char}) of degradable strongly entanglement breaking (SEB) maps  on Hilbert spaces (finite or infinite dimensions), which includes a structure theorem. Theorem \ref{thm-DEB-char} also answers the question of when a map complementary to a SEB-map is SEB. Theorem \ref{thm-DEB-char} and \cite[Corollary 7]{MuSi22} together characterizes all degradable PPT-maps on matrix algebras (See Remark \ref{rmk-degr-PPT}). In Theorem \ref{thm-Schur-SEB-antideg}, we provide a class of positive maps for which the notions of antidegradabilty and strongly entanglement breaking coincides.   We concentrate on the finite-dimensional case in Section 4 and prove (Theorem \ref{thm-degr-Cstar-ext}) that the class of unital degradable entanglement breaking maps coincide with the $C^*$-extreme points of all unital entanglement breaking maps introduced by Bhat et al. \cite{BDMS23}, and they are in turn adjoints of extreme QC-channels (\cite{HSR03}). In Theorem \ref{thm-PPT-proj}, we characterize PPT-channels for which the associated Choi-matrix is a projection and their degradability property. Finally, we consider degradable completely positive maps too.  Though $C^*$-extreme points of unital completely positive maps are degradable, we see  in Example \ref{eg-counter-CP-extr-degr} that the converse is not generally true.

\section{Preliminaries}

 Through out $\mcl{H}$ and $\mcl{K}$ (possibly with suffices) denote separable (finite or infinite dimensional)  complex Hilbert spaces with inner products linear in the second variable and conjugate linear in the first variable. We let $\B{\mcl{H},\mcl{K}}$ denotes the space of all bounded linear operators from $\mcl{H}$ to $\mcl{K}$, and $\B{\mcl{H}}^+$ denotes the cone of all positive operators in $\B{\mcl{H}}:=\B{\mcl{H},\mcl{H}}$. Recall that $A\in\B{\mcl{H}}^+$ if and only if $A=B^*B$ for some $B\in\B{\mcl{H}}$ if and only if $\ip{x,Ax}\geq 0$ for all $x\in\mcl{H}$. We let $\TC{\mcl{H}}$ denotes the space of all trace-class operators on $\mcl{H}$, which is a Banach space w.r.t the trace-norm given by $\norm{T}_{1}:=\tr(\abs{T})$, where $\tr(\abs{T})$ is the trace of $\abs{T}:=(T^*T)^{\frac{1}{2}}$  for all $T\in\TC{\mcl{H}}$. The space $(\B{\mcl{H}}, \norm{.})$ is isometrically isomorphic to the dual space of $(\TC{\mcl{H}}, \norm{.}_1)$ via the map $A\mapsto \phi_{A}$, where $\phi_{A}(T):=\tr(AT)$ for all $T\in\TC{\mcl{H}}$ and $A\in\B{\mcl{H}}$. 
 The elements of 
 \begin{align*}
   \ST{\mcl{H}}  :=\{\rho \in \TC{\mcl{H}}: \rho \in\B{\mcl{H}}^+, \norm{\rho}_1 = 1\}  
 \end{align*}
 are called \it{states} on $\mcl{H}$. A state $\rho \in \ST{\mcl{H}_{1}\otimes\mcl{H}_{2}}$ is said to be
 \begin{enumerate}[label=(\roman*)]
     \item \it{separable} (\cite{Wer89}) if $\rho$ belongs to the closure (w.r.t $\norm{\cdot}_1$) of the convex hull of the set  
     $$\{\rho^{(1)}\otimes \rho^{(2)}: \rho^{(j)}\in\ST{\mcl{H}_j}, j=1,2\};$$
     \item \it{countably decomposable separable} (\cite{KSW05}) if there exist states $\rho_n^{(j)}\in\ST{\mcl{H}_j},j=1,2$ such that\footnote{Through out this article we let $\Lambda\subseteq\mbb{N}$ denotes an index set. Given a family $\{A_j\}_{j\in\Lambda}$ of bounded linear operators on a Hilbert space $\ssum_{j\in\Lambda}A_j$ and $\sum_{j\in\Lambda}A_j$ denote the limit of the series in the strong operator topology (SOT) and trace-norm topology, respectively.}
     \begin{align}\label{eq-CD-sep}
         \rho = \sum_{n=1}^{N}p_n\rho_n^{(1)} \otimes \rho_n^{(2)},
     \end{align}
     where $p_n\in[0,1]$ with $\sum_np_n=1$ and $N\in\mbb{N}\cup\{\infty\}$. 
 \end{enumerate}
Note that, by using spectral theorem for compact positive operators, we can replace states by rank-one states in the definition of separability and countably decomposable separability. Every countably decomposable separable state is separable, but the converse is not true in general (\cite{KSW05}). If both $\mcl{H}_1$ and $\mcl{H}_2$ are finite-dimensional spaces, then separability and countably decomposable separability coincide; in such cases $N<\infty$ in \ref{eq-CD-sep} (c.f. \cite{ChDo13}). We say that a positive operator $T\in\TC{\mcl{H}_1\otimes\mcl{H}_2}$ is separable (resp, countably decomposable separable) if $\frac{1}{\tr(T)}T\in\ST{\mcl{H}_1\otimes\mcl{H}_2}$ is separable (resp, countably decomposable separable).

 A linear map $\Phi:\B{\mcl{H}_1}\to\B{\mcl{H}_2}$ is said to be \it{positive} if $\Phi(\B{\mcl{H}_1}^+)\subseteq\B{\mcl{H}_2}^+$; \it{completely positive} (CP)  if the map $\id_k \otimes \Phi:\M{k}\otimes \B{\mcl{H}_{1}}\to \M{k}\otimes \B{\mcl{H}_{2}}$ is positive for all $k\in\mbb{N}$, where $\id_k$ denotes the identity map on $\M{k}=\B{\mbb{C}^k}$; \it{co-completely positive} (co-CP) if the map $\T \otimes \Phi:\M{k}\otimes \B{\mcl{H}_{1}}\to \M{k}\otimes \B{\mcl{H}_{2}}$ is positive for all  $k \in \mbb{N}$, where $\T$ denotes the transpose map on $\M{k}$. If $\Phi$ is both CP and co-CP, then it is called a \it{positive partial transpose} (PPT) map. Note that, given $V\in\B{\mcl{H}_2,\mcl{H}_1}$, the map $\mathrm{Ad_V}:\B{\mcl{H}_1}\to\B{\mcl{H}_2}$ defined by $\mathrm{Ad}_V(X):=V^*XV$ is a normal (i.e., ultra weak continuous) CP-map.   It is well-known (\cite{Sti55},\cite{Kra71}) that given a positive linear map $\Phi:\B{\mcl{H}_1}\to\B{\mcl{H}_2}$ the following assertions are equivalent:
 \begin{enumerate}[label=(\roman*)]
    \item $\Phi$ is a normal CP-map. 
    \item (Kraus decomposition:) There exist a countable family $\{A_j\}_{j\in \Lambda}\subseteq\B{\mcl{H}_2,\mcl{H}_1}$, called Kraus operators, such that $\ssum_{j\in\Lambda}A_j^*A_j\in\B{\mcl{H}_2}$  and
    \begin{align}\label{eq-Kraus-decomp-CP}
        \Phi(X)=\ssum_{j\in \Lambda} A_j^*XA_j,\qquad\forall~X\in\B{\mcl{H}_1}.
    \end{align}
    \item (Stinespring representation:) There exist a pair $(\mcl{K},A)$ consisting of a  separable  Hilbert space $\mcl{K}$ and a bounded linear operator $A:\mcl{H}_{2} \to  \mcl{H}_1 \otimes\mcl{K}$ such that 
    \begin{align}\label{eq-Sti-decomp-CP}
       \Phi(X)= A^*(X\otimes I_{\mcl{K}})A,\qquad\forall~X\in\B{\mcl{H}_1},
    \end{align}
    where $I_{\mcl{K}}$ is the identity operator on $\mcl{K}$.
 \end{enumerate}
 The pair $(\mcl{K},A)$ in the Stinespring representation can be chosen to be minimal in the sense that $\cspan\{(X\otimes I_\mcl{K})Ah: h\in\mcl{H}_2, X\in\B{\mcl{H}_1}\}=\mcl{H}_1\otimes\mcl{K}$; such a pair is unique up to unitary equivalence. If $\Phi$ has Kraus decomposition \eqref{eq-Kraus-decomp-CP}, then by taking $\mcl{K}$ to be a Hilbert space with an orthonormal basis $\{e_j:j\in \Lambda\}$  and by defining $A:\mcl{H}_{2} \to  \mcl{H}_1 \otimes\mcl{K}$ as  $Ax:=\sum_{j\in \Lambda} A_jx\otimes e_j$ we get a Stinespring representation \eqref{eq-Sti-decomp-CP}. Conversely, if $\Phi$ has a Stinespring representation \eqref{eq-Sti-decomp-CP}, fixing an orthonormal basis $\{e_j\}_{j\in \Lambda}$ of $\mcl{K}$, the operators  $A_j:\mcl{H}_2\to\mcl{H}_1$ defined by $A_j^*x:=A^*(x\otimes e_j),j\in\Lambda$, gives a Kraus decomposition \eqref{eq-Kraus-decomp-CP} of $\Phi$.

 Let $\Phi: \TC{\mcl{H}_{1}} \to \TC{\mcl{H}_{2}}$ be a  bounded (w.r.t. trace norm) linear map. Then there exists a  bounded (w.r.t. operator norm) linear map $\Phi^{*}: \B{\mcl{H}_{2}} \to \B{\mcl{H}_{1}}$, called the \it{dual} of $\Phi$, uniquely determined by the identity  $\tr({\Phi^*(X)}T)= \tr(X\Phi(T))$ for $T\in\TC{\mcl{H}},X\in\B{\mcl{H}_2}$. A bounded linear map $\Phi: \TC{\mcl{H}_{1}} \to \TC{\mcl{H}_{2}}$ is said to be a \it{quantum operation} if $\Phi^*$ is a normal CP-map. (In quantum information theory, one often assumes that $\tr(\Phi(T))=\tr(T)$ for all $T\in\TC{\mcl{H}_1}$ (equivalently $\Phi^*(I)=I$), and in such case $\Phi$ is called a \it{quantum channel}). Given a bounded linear map $\Phi:\TC{\mcl{H}_1}\to\TC{\mcl{H}_2}$,  from the duality between the maps $\Phi$ and $\Phi^*$, it follows that the following assertions are equivalent: 
 \begin{enumerate}[label=(\roman*)] 
    \item $\Phi$ is a quantum operation.
    \item (Kraus decomposition:) There exist a countable family of  bounded linear operators $\{A_j\}_{j\in\Lambda}\subseteq\B{\mcl{H}_1,\mcl{H}_2}$ with  $\ssum_{j\in \Lambda}A_j^*A_j\in \B{\mcl{H}_1}$ such that \footnote{Whenever the series $\sum_jA_j^*A_j$ converges in weak operator topology the series $\sum_jA_jTA_j^*$ converges in trace norm for every  $T\in\TC{\mcl{H}_1}$. (c.f. \cite[Proposition 6.3]{Att6}).}
    \begin{align}\label{eq-Kraus-decomp-QC}
        \Phi(T)=\sum_{j\in \Lambda} A_jTA_j^*,\qquad\forall~T\in\TC{\mcl{H}_1}.
    \end{align}
    \item (Stinespring representation:) There exist a pair $(\mcl{K},A)$ consisting of a separable Hilbert space $\mcl{K}$ and a bounded operator $A:\mcl{H}_1 \to \mcl{H}_2\otimes\mcl{K}$ such that 
  \begin{align}\label{eq-Sti-decomp-QC}
       \Phi(T)= \tr_{\mcl{K}}(ATA^*),\qquad\forall~T\in\TC{\mcl{H}_1},
   \end{align}
   where $\tr_\mcl{K}$ is the partial trace w.r.t. the Hilbert space $\mcl{K}$.
 \end{enumerate}
 If $\Phi$ is a quantum channel, then in $(ii)$ we get $\ssum_jA_j^*A_j=I$; and in $(iii)$  we have $A$ is an isometry. Observe that $(\mcl{K},A)$ is a Stinespring representation of a quantum operation $\Phi$ if and only if $(\mcl{K},A)$ is a Stinespring representation of  $\Phi^*$.

 Note that a bounded normal linear map $\Phi:\B{\mcl{H}_1}\to\B{\mcl{H}_2}$ is CP if and only if $\id_{\mcl{K}}\otimes\Phi:\B{\mcl{K}\otimes\mcl{H}_1}\to\B{\mcl{K}\otimes\mcl{H}_2}$ is positive for all separable Hilbert spaces $\mcl{K}$.

\begin{defn}
 A quantum operation $\Phi:\TC{\mcl{H}_{1}} \to \TC{\mcl{H}_{2}}$ is said to be 
 \begin{enumerate}[label=(\roman*)]
    \item \it{entanglement breaking (EB-)map} if $(\id_\mcl{K} \otimes \Phi)(\rho)\in\TC{\mcl{K}\otimes\mcl{H}_2}$ is separable for all separable Hilbert space $\mcl{K}$ and $\rho \in \ST{\mcl{K} \otimes \mcl{H}_{1}}$.
    \item \it{strong entanglement breaking (SEB-)map} if $(\id_{\mcl{K}} \otimes \Phi)(\rho)\in\TC{\mcl{K}\otimes\mcl{H}_2}$ is countably decomposable separable for all separable Hilbert space $\mcl{K}$ and $\rho \in \ST{\mcl{K}\otimes \mcl{H}_{1}}$.
    \item \it{positive partial transpose (PPT)} if its dual $\Phi^*:\B{\mcl{H}_2} \to \B{\mcl{H}_1}$ is a PPT-map. 
    \end{enumerate}
\end{defn}

 Every SEB-map is an EB-map; the converse is true if $\mcl{H}_1,\mcl{H}_2$ are finite dimensional, but not true in general. See \cite{Kho08,HSR03,KSW05,He13,LiDu15} for more details. We let $\mathrm{SEB}(\mcl{H}_1,\mcl{H}_2)$ denotes the set of all SEB-maps from $\TC{\mcl{H}_1}$ to $\TC{\mcl{H}_2}$. Note that if $\Phi\in\mathrm{SEB}(\mcl{H}_1,\mcl{H}_2)$, then $\Gamma\circ\Phi$ is SEB for every quantum operation $\Gamma:\TC{\mcl{H}_2}\to\TC{\mcl{K}}$. From the following theorem it follows that $\Phi\circ\Gamma$ is also SEB. 
  
\begin{thm}\label{thm-SEB-char}
 Given a quantum operation $\Phi:\TC{\mcl{H}_1} \to \TC{\mcl{H}_2}$ the following statements are equivalent: 
 \begin{enumerate}[label=(\roman*)]
    \item $\Phi$ is SEB.
    \item (Kraus decomposition:) There exist rank-one Kraus operators $\{A_j\}_{j\in \Lambda}\subseteq\B{\mcl{H}_1,\mcl{H}_2}$ with $\ssum_{j\in\Lambda}A_j^*A_j\in\B{\mcl{H}_1}$ such that 
    \begin{align}\label{eq-Kraus-decomp-SEB}
        \Phi(T)=\sum_{j\in\Lambda}A_jTA_j^*,\qquad\forall~T\in\TC{\mcl{H}_1}.
    \end{align}
    \item (Holevo form:) There exist $\{R_j\}_{j\in \Lambda}\subseteq\ST{\mcl{H}_2}$ and $\{F_j\}_{j\in \Lambda}\subseteq\B{\mcl{H}_1}^+$ with $\ssum_{j\in \Lambda} F_{j}\in\B{\mcl{H}_1}$ such that 
    \begin{align}\label{eq-Holevo-SEB}
        \Phi(T)= \sum_{j\in\Lambda}\tr(TF_j)R_j,\qquad\forall~T\in\TC{\mcl{H}_1}.
    \end{align}
    \end{enumerate}
\end{thm}

 For the finite-dimensional case of the above theorem we refer \cite{HSR03}. See \cite{KSW05,He13,LiDu15} for a proof of the above theorem for quantum channels on infinite-dimensional Hilbert spaces; the theorem is  true even for quantum operations. See Appendix (Theorem \ref{thm-QO-char}) for details. 

\begin{defn}
 Let  $\Phi:\TC{\mcl{H}_1}\to\TC{\mcl{H}_2}$ and $\Psi:\TC{\mcl{H}_1}\to\TC{\mcl{H}_3}$ be two quantum operations. Then $\Psi$ is said to be \it{complementary} to $\Phi$ If there exists a bounded operator $A:\mcl{H}_{1} \to \mcl{H}_2\otimes\mcl{H}_{3}$  such that 
 \begin{equation}\label{eq-compl-cp}
   \Phi(T)=\tr_{\mcl{H}_3}(ATA^*)
   \qquad
   \mbox{and}
   \qquad
   \Psi(T)=\tr_{\mcl{H}_2}(ATA^*) 
 \end{equation}
 for all $T \in \TC{\mcl{H}_1}$. We  let $\msc{CM}(\Phi)$ denotes the set of all quantum operations complementary to $\Phi$.  We say that $\Phi$ is \it{self-complementary} if $\Phi$ is complementary to itself (i.e., $\Phi\in\msc{CM}(\Phi)$). 
\end{defn}

 Let $\Phi:\TC{\mcl{H}_1}\to\TC{\mcl{H}_2}$ and $\Psi:\TC{\mcl{H}_1}\to\TC{\mcl{H}_3}$ be two quantum operations. Assume $\Psi\in\msc{CM}(\Phi)$ and $A:\mcl{H}_1\to\mcl{H}_2\otimes\mcl{H}_3$ be such that \eqref{eq-compl-cp} holds. Then taking $\u{A}=FA:\mcl{H}_1\to\mcl{H}_3\otimes\mcl{H}_2$, where $F:\mcl{H}_2\otimes\mcl{H}_3\to \mcl{H}_3\otimes\mcl{H}_2$ is the unitary operator given by $F(x\otimes y)=y\otimes x$ for all $x\in\mcl{H}_2,y\in\mcl{H}_3$, we get $\Psi(T)=\tr_{\mcl{H}_2}(\u{A}T\u{A}^*)$ and $\Phi(T)=\tr_{\mcl{H}_3}(\u{A}T\u{A}^*)$ for all $T\in\TC{\mcl{H}_1}$. Thus $\Phi$ is complementary to $\Psi$.  Thus $\Psi\in\msc{CM}(\Phi)$ if and only if $\Phi\in\msc{CM}(\Psi)$.
 
 If the quantum operation $\Phi:\TC{\mcl{H}_1}\to\TC{\mcl{H}_2}$ has a Stinespring representation $(\mcl{K},A)$, then the map $\Phi^c:\TC{\mcl{H}_1}\to\TC{\mcl{K}}$ defined by 
  \begin{align}\label{eq-compl-defn-1}
         \Phi^c(T):=\tr_{\mcl{H}_2}(ATA^*),\qquad\forall~T\in\TC{\mcl{H}_1}
  \end{align}
 is a quantum operation and is complementary to $\Phi$.  Note that the definition of the map $\Phi^c$ depends on the representation $(\mcl{K},A)$, so we may use the notation $\Phi^c_{(\mcl{K},A)}$ in place of $\Phi^c$. Letting $\u{A}=FA$, where $F:\mcl{H}_2\otimes\mcl{K}\to\mcl{K}\otimes\mcl{H}_2$ is the flip map as above,  we have $(\mcl{H}_2,\u{A})$ is a  Stinespring representation for $\Psi:=\Phi^c\in\msc{CM}(\Phi)$. So $ \Psi^c_{(\mcl{H}_2,\u{A})}:\TC{\mcl{H}_1}\to\TC{\mcl{H}_2}$ is given by 
 \begin{align*}
     \Psi^c_{(\mcl{H}_2,\u{A})}(T)
        =\tr_{\mcl{K}}(\u{A}T\u{A}^*)
        =\tr_{\mcl{K}}(ATA^*)=\Phi(T),\qquad\forall~T\in\TC{\mcl{H}_1}.
 \end{align*}
 Given a Stinespring representation $(\mcl{K},A)$ of a  quantum operation  $\Phi:\TC{\mcl{H}_1}\to\TC{\mcl{H}_2}$ we always let $(\Phi^c)^c:=(\Phi^c_{(\mcl{K},A)})^c_{(\mcl{H}_2,\u{A})}$. Then from the above equation we have $(\Phi^c)^c=\Phi$. Now, suppose $\Phi:\TC{\mcl{H}_1}\to\TC{\mcl{H}_2}$ has a Kraus decomposition as in \eqref{eq-Kraus-decomp-QC}. Let $\mcl{K}$ be a separable complex Hilbert space  with orthonormal basis $\{e_j\}_{j\in\Lambda}$. Note that $A:\mcl{H}_1\to\mcl{H}_2\otimes\mcl{K}$ defined by $A(x)=\sum_{j\in\Lambda}A_jx\otimes e_j$ is a bounded operator such that $\Phi(T)=\tr_{\mcl{K}}(ATA^*)$ so that  $(\mcl{K},A)$ is a Stinespring representation for $\Phi$. Further, $\Phi^c$ defined in \eqref{eq-compl-defn-1} is given by
  \begin{align}\label{eq-compl-defn-2}
     \Phi^c(T):=\tr_{\mcl{H}_2}(ATA^*)=\sum_{i,j\in\Lambda}\tr(A_iTA_j^*)\ranko{e_i}{e_j}.
 \end{align}
 (See Proposition \ref{Comp_channel_formula} for details.)
     
 Note that if $\Psi\in\msc{CM}(\Phi)$, then $\mathrm{Ad}_W\circ\Psi\in\msc{CM}(\Phi)$ for all co-isometry $W$ between appropriate Hilbert spaces. 
 This fact follows from the following identities: 
 \begin{align*}
     \tr_{\mcl{H}_3}\Big((I_{\mcl{H}_2}\otimes Y)ATA^*(I_{\mcl{H}_2}\otimes Y^*)\Big)&=\tr_{\mcl{K}}((I_{\mcl{H}_2}\otimes Y^*Y)ATA^*)\\
      \tr_{\mcl{H}_2}\Big((I_{\mcl{H}_2}\otimes Y)ATA^*(I_{\mcl{H}_2}\otimes Y^*)\Big)&=Y\tr_{\mcl{H}_2}(ATA^*)Y^*
 \end{align*}
 for every $A\in\B{\mcl{H}_1,\mcl{H}_2\otimes\mcl{K}}, T\in\TC{\mcl{H}_1}$ and $Y\in\B{\mcl{K},\mcl{H}_3}$. It follows now that $(\Phi\circ \mathrm{Ad}_Z)_{(\mcl{K},AZ^*)}^c=\Phi^c_{(\mcl{K},A)}\circ\mathrm{Ad}_Z$ for all $Z\in\B{\mcl{H},\mcl{H}_1}$ and $(Ad_V\circ\Phi)^c_{(\mcl{K},(V^*\otimes I_{\mcl{K}})A)}=\Phi^c_{(\mcl{K},A)}$ for all co-isometry $V\in\B{\mcl{H}_3,\mcl{H}_2}$. 

\begin{rmk}\label{rmk-unique-complem}
 Let $\Phi:\TC{\mcl{H}_1}\to\TC{\mcl{H}_2}$ be a quantum operation with a minimal Stinespring representation $(\widehat{\mcl{K}},\widehat{A})$. 
 \begin{enumerate}[label=(\roman*)]
     \item Suppose $(\widetilde{\mcl{K}},\widetilde{A})$ is another minimal Stinespring representation of $\Phi$ so that, by Theorem \ref{thm-Stine-QC-uniq}, there is a unitary $U:\widehat{\mcl{K}}\to\widetilde{\mcl{K}}$ such that $\widetilde{A}=(I_{\mcl{H}_2}\otimes U)\widehat{A}$. Then 
     \begin{align*}
        \Phi^c_{(\widetilde{\mcl{K}},\widetilde{A})}(T)
           =\tr_{\mcl{H}_2}(\widetilde{A}T\widetilde{A}^*)
          =\mathrm{Ad}_U\circ\Phi^c_{(\widehat{\mcl{K}},\widehat{A})}(T),\qquad\forall~T\in\TC{\mcl{H}_1}.
     \end{align*}
     Thus the complement of $\Phi$ defined as in \eqref{eq-compl-defn-1} from a minimal Stinespring representation is unique up to unitary conjugation, and we denote such a complementary map by $\Phi^c_{\min}$. 
     \item  Let $\Psi:\TC{\mcl{H}_1}\to\TC{\mcl{K}}$ be a quantum operation complementary to $\Phi$. So there exists $A\in\B{\mcl{H}_1,\mcl{H}_2\otimes\mcl{K}}$ such that $\Phi(T)=\tr_{\mcl{K}}(ATA^*)$ and $\Psi(T)=\tr_{\mcl{H}_2}(ATA^*)$ for all $T\in\TC{\mcl{H}_1}$. Since $(\mcl{K},A)$ is a Stinespring representation for $\Phi$, from Theorem  \ref{thm-Stine-QC-uniq}, there exists an isometry $V\in\B{\widehat{\mcl{K}},\mcl{K}}$ such that  $A=(I_{\mcl{H}_2}\otimes V)\widehat{A}$ and $\widehat{A}=(I_{\mcl{H}_2}\otimes V^*)A$.  Consequently, for all $T\in\TC{\mcl{H}_1}$ we have
 \begin{align*}
      \Psi(T)=V\Phi^c_{\min}(T)V^*
      \quad\mbox{and}\quad
      \Phi^c_{\min}(T)=V^*\Psi(T)V.
 \end{align*}
 Thus we have
 \begin{align*}
      \msc{CM}(\Phi)=\{\mathrm{Ad}_{V^*}\circ\Phi^c_{\min}: V\in\B{\widehat{\mcl{K}},\mcl{K}} \mbox{ is an isometry and } \mcl{K}\mbox{ Hilbert space}\}
 \end{align*}
 In particular, $\Phi$ is self-complementary if and only if $\Phi=\mathrm{Ad}_{V^*}\circ\Phi^c_{\min}$ for some isometry $V\in\B{\widehat{\mcl{K}},\mcl{H}_2}$.  Now suppose $\Psi_1:\TC{\mcl{H}_1}\to\TC{\mcl{H}_3}$ and $\Psi_2:\TC{\mcl{H}_1}\to\TC{\mcl{H}_4}$ are two quantum operations complementary to $\Phi$. Then there exist isometries $V_1\in\B{\widehat{\mcl{K}},\mcl{H}_3}$ and $V_2\in\B{\widehat{\mcl{K}},\mcl{H}_4}$ such that $\Phi^c_{\min}=\mathrm{Ad}_{V_j}\circ\Psi_j$ and $\Psi_j=\mathrm{Ad}_{V_j^*}\circ\Phi^c_{\min}$ for $j=1,2$.  Note that $W:=V_2V_1^*\in\B{\mcl{H}_3,\mcl{H}_4}$ is a partial isometry such that 
  \begin{align}\label{eq-compl-channel-equiv}
      \Psi_2(T)=W\Psi_1(T)W^*
      \quad\mbox{and}\quad
      \Psi_1(T)=W^*\Psi_2(T)W,\qquad\forall~T\in\TC{\mcl{H}_1}.
  \end{align}  
   Thus, maps complementary to a quantum operation is unique up to a  partial isometry in the sense of \eqref{eq-compl-channel-equiv}. In finite-dimensional case uniqueness is up to an isometry. See Section \ref{sec-fd-case}.  
 \end{enumerate}
\end{rmk}
 
\begin{defn}(\cite{DeSh05})
 A quantum operation $\Phi:\TC{\mcl{H}_1}\to\TC{\mcl{H}_2}$ is said to be 
 \begin{enumerate}[label=(\roman*)]
    \item \it{degradable}  if there exists a quantum operation $\Gamma:\TC{\mcl{H}_2}\to\TC{\mcl{K}}$ for some Hilbert space $\mcl{K}$ such that $\Gamma\circ\Phi\in \msc{CM}(\Phi)$. 
    \item \it{anti-degradable} if there exists a $\Psi\in \msc{CM}(\Phi)$ such that $\Psi$ is degradable.
\end{enumerate}
\end{defn}

 Suppose $(\mcl{K},A)$ Stinespring representation for the quantum operation $\Phi:\TC{\mcl{H}_1}\to\TC{\mcl{H}_2}$  and consider $\Phi^c:\TC{\mcl{H}_1}\to\TC{\mcl{K}}$ as in \eqref{eq-compl-defn-1}. Then from the above remark  it follows that $\Phi$ is degradable if and only if there exists a quantum operation $\Gamma:\TC{\mcl{H}_2}\to\TC{\mcl{K}}$ such that $\Gamma\circ\Phi=\Phi^c$. Similarly, $\Phi$ is anti-degradable if and only if $\Phi^c$ is degradable if and only if there exists a quantum operation $\Gamma:\TC{\mcl{K}}\to\TC{\mcl{H}_2}$ such that $\Gamma\circ\Phi^c=\Phi$. 

\begin{eg}
 Let $A\in\B{\mcl{H}}$ and consider $\Phi=\mathrm{Ad}_A:\TC{\mcl{H}}\to\TC{\mcl{H}}$. Note that $\Phi^c(T)=\tr(A^*TA)=\Gamma\circ\Phi(T)$ for all $T\in\TC{\mcl{H}}$, where $\Gamma(\cdot)=\tr(\cdot)$. Thus $\Phi$ is degradable. But  $\Phi$ anti-degradable if and only if  $\Phi$ is SEB. For, if $\Phi$ is anti-degradable there exists a quantum operation $\Gamma':\mbb{C}\to\TC{\mcl{H}}$ such that $\Phi=\Gamma'\circ\Phi^c$. So there exists a positive operator $R\in\TC{H}$ such that $\Phi(T)=\tr(A^*TA)R=\tr(AA^*T)R$ for all $T\in\TC{\mcl{H}}$. Thus $\Phi$ is SEB (so that $A$ is a rank-one operator). Converse follows from the following remark. 
\end{eg} 

\begin{rmk}
 Let $\Phi\in\mathrm{SEB}(\mcl{H}_1,\mcl{H}_2)$. Assume $\Phi$ has a Kraus decomposition as in \ref{eq-Kraus-decomp-SEB} with  $A_j=\ranko{v_j}{u_j}$, where $u_j\in\mcl{H}_1$ and $v_j\in\mcl{H}_2$. Let $\mcl{K}$ be a separable complex Hilbert space with orthonormal basis $\{e_j:j\in\Lambda\}$ and let $B_j:=\ranko{\frac{v_j}{\norm{v_j}}}{e_j}$ for all $j\in\Lambda$. Define $\Gamma:\TC{\mcl{K}}\to\TC{\mcl{H}_2}$  by $\Gamma(T):=\sum_{j\in\Lambda}B_jTB_j^*$  for all $T\in\TC{\mcl{H}_1}$. Since $\ssum_j B_j^*B_j=I$ the map $\Gamma$ is a quantum channel. Further, 
 \begin{align*}
    \Gamma\circ\Phi^c(T)
                      =\sum_{i,j\in\Lambda}\ip{u_i,Tu_j}\ip{v_j,v_i}\Gamma(\ranko{e_i}{e_j})
                      =\sum_{i\in\Lambda}\ip{u_i,Tu_i}\ranko{v_i}{v_i}
                      =\Phi(T)
 \end{align*}
 for all $T\in\TC{\mcl{H}_1}$, where $\Phi^c$ is given by \eqref{eq-compl-defn-2}. This concludes that every SEB-map is anti-degradable. This fact was observed for the finite dimensional case in \cite{CRS08}. 
\end{rmk}

\section{Degradable SEB-maps} 

 In this section, our aim is to determine all degradable SEB-maps. We begin with an example.
 
\begin{eg}
 Consider the map $\Phi:\TC{\mcl{H}_1}\to\TC{\mcl{H}_2}$ defined by 
\begin{align*}
    \Phi(T)=\ip{u,Tu}R,\qquad\forall~T\in\TC{\mcl{H}_1},
\end{align*}
 where $u\in\mcl{H}_1$ is a non-zero vector and $R\in\TC{\mcl{H}_2}$ is positive. By spectral theorem of compact positive operators there exist orthogonal vectors $\{v_j\}_{j\in\Lambda}\subseteq \mcl{H}_2$ such that $R=\sum_{j\in\Lambda}\ranko{v_j}{v_j}$. Then $\Phi(T)=\sum_{j\in\Lambda}A_jTA_j^*$ for all $T\in\TC{\mcl{H}_1}$, where $A_j=\ranko{v_j}{u}$. Let $\mcl{K}$ be a  Hilbert space with an orthonormal basis $\{e_j\}_{j\in\Lambda}$ and consider $\Phi^c$ as in \eqref{eq-compl-defn-2}. Since $\widetilde{R}:=\sum_{j\in\Lambda}\ip{v_j,v_j}\ranko{e_j}{e_j}\in\TC{\mcl{K}}$ is positive the complementary map
 \begin{align*}
     \Phi^c(T)=\sum_{i,j\in\Lambda}\ip{u,Tu}\ip{v_j,v_i}\ranko{e_i}{e_j}  
            =\ip{u,Tu}\widetilde{R},\qquad\forall~T\in\TC{\mcl{H}_1}
 \end{align*}
 is SEB. Hence $\Phi^c$ anti-degradable so that $\Phi$ is degradable.
\end{eg}

\begin{thm}\label{thm-DEB-char}
 Given $\Phi\in \mathrm{SEB}(\mcl{H}_1,\mcl{H}_2)$ the following conditions are equivalent:
 \begin{enumerate}[label=(\roman*)]
    \item $\Phi$ is degradable. 
    \item Any (and, hence all) of the elements of $\msc{CM}(\Phi)$ is SEB.
    \item Any (and, hence all) of the elements of $\msc{CM}(\Phi)$ is PPT.
    \item There exist non-zero vectors  $\{u_j\}_{j\in\Lambda}\subseteq\mcl{H}_1$ and unit vectors  $\{v_j\}_{j\in\Lambda}\subseteq\mcl{H}_2$ with $\ssum_{j\in\Lambda}\ranko{u_j}{u_j}\in\B{\mcl{H}_1}$ and  $\ip{v_i,v_j}=0$ whenever  $\{u_i,u_j\}$ is linearly independent such that  
    \begin{align}\label{eq-Holevo-SEB-degrd}
         \Phi(T)=\sum_{j\in \Lambda}\ip{u_j,Tu_j}\ranko{v_j}{v_j},\qquad\forall~T\in\TC{\mcl{H}_1}. 
    \end{align}
    \item There exist rank-one positive operators $\{F_j\}_{j\in\Lambda}\subseteq\TC{\mcl{H}_1}$ with $\ssum_{j\in\Lambda} F_j\in\B{\mcl{H}_1}$ and states $\{R_j\}_{j\in\Lambda}\subseteq\ST{\mcl{H}_2}$ with $R_iR_j=0$ for all $i\neq j$ such that 
    \begin{align*}
        \Phi(T)=\sum_{j\in\Lambda}\tr(TF_j)R_j,\qquad\forall~T\in\TC{\mcl{H}_1}. 
    \end{align*}
    \item $\Phi$ is self-complementary.
 \end{enumerate}
\end{thm}

\begin{proof}
 $(i)\Rightarrow (ii)$ Assume $\Phi$ is degradable so that $\Phi^c_{\min}=\Gamma\circ\Phi$ for some quantum operation $\Gamma$. Since $\Phi$ is SEB   we have $\Phi^c_{\min}$ is also SEB. Now given any $\Psi\in\msc{CM}({\Phi})$ by Remark \ref{rmk-unique-complem}, there exists an isometry $V$ such that $\Psi=\mathrm{Ad}_{V^*}\circ \Phi^c_{\min}$, hence $\Psi$ is also SEB.\\
 $(ii)\Rightarrow (iii)$ Follows trivially.\\ 
 $(iii)\Rightarrow (iv)$ Let  $A_j\in\B{\mcl{H}_1,\mcl{H}_2}$ be rank-one operators with $\ssum_jA_j^*A_j\in\B{\mcl{H}_1}$ such that  $\Phi(\cdot)=\sum_{j\in\Lambda}A_j(\cdot)A_j^*$. Assume $A_j=\ranko{v_j}{u_j}$ for some $u_j\in\mcl{H}_1$ and unit vector $v_j\in\mcl{H}_2$ for all $j\in\Lambda$. Then $\ssum_j\ranko{u_j}{u_j}\in\B{\mcl{H}_1}$ and $\Phi(T)=\sum_{j\in\Lambda}\ip{u_j,Tu_j}\ranko{v_j}{v_j}$ for all $T\in\TC{\mcl{H}_1}$. Let $\mcl{K}$ be a Hilbert space with orthonormal basis $\{e_j\}_{j\in\Lambda}$ and consider the complementary map $\Phi^c:\TC{\mcl{H}_1}\to\TC{\mcl{K}}$ as in \eqref{eq-compl-defn-2}, i.e.,
 \begin{align*}
     \Phi^c(T) 
          =\sum_{i,j\in\Lambda} \ip{v_j,v_i}\ip{u_i,Tu_j}\ranko{e_i}{e_j},\qquad\forall~T\in\TC{\mcl{H}_1}.
 \end{align*}
 Fix $i\neq j\in\Lambda$ and assume that $\{u_i,u_j\}$ is linearly independent. Consider the standard orthonormal basis $\{e_1^{(2)},e_2^{(2)}\}$ for $\mbb{C}^2$. Let 
 \begin{align*}
     A_{ij}:=\ranko{e_i}{e_1^{(2)}}+\ranko{e_j}{e_2^{(2)}}\in\B{\mbb{C}^2,\mcl{K}}
 \end{align*}
 and $B_{ij}\in\B{\mcl{H}_1,\mbb{C}^2}$ be such that $B_{ij}u_i=e_1^{(2)},B_{ij}u_j=e_2^{(2)}$
 and $B_{ij}=0$ on the orthogonal complement of $\lspan\{u_i,u_j\}$.  Then for all $T\in\M{2}$ we have
 \begin{align*}
     \mathrm{Ad}_{A_{ij}}\circ\Phi^c\circ\mathrm{Ad}_{B_{ij}}(T)&=\sum_{k,l\in\Lambda}\ip{v_l,v_k}\ip{B_{ij}u_k,TB_{ij}u_l}\ranko{A_{ij}^*e_k}{A_{ij}^*e_l}\\
     &=\sum_{k,l\in\{i,j\}}\ip{v_l,v_k}\ip{B_{ij}u_k,TB_{ij}u_l}\ranko{A_{ij}^*e_k}{A_{ij}^*e_l}\\
     &=\Matrix{\norm{v_i}^{2}\ip{e_1^{(2)},Te_1^{(2)}}  & \ip{v_j,v_i}\ip{e_1^{(2)},Te_2^{(2)}} \\ \ip{v_i, v_j}\ip{e_2^{(2)},Te_1^{(2)}} & \norm{v_j}^{2}\ip{e_1^{(2)},Te_2^{(2)}}}\\
     &=\Matrix{\norm{v_i}^{2}  & \ip{v_j,v_i} \\ \ip{v_i, v_j} & \norm{v_j}^{2}}\odot T,
 \end{align*}
 where $\odot$ denotes the Schur product. Since $\Phi^c$ is PPT, the composition $\mathrm{Ad}_{A_{ij}}\circ\Phi^c\circ\mathrm{Ad}_{B_{ij}}:\M{2}\to\M{2}$ is PPT. This is possible if and only if $\sMatrix{\norm{v_i}^{2}  & \ip{v_j,v_i} \\ \ip{v_i, v_j} & \norm{v_j}^{2}}$ is a diagonal matrix (c.f. \cite[Proposition 4.1]{KMP18}). Therefore   $\ip{v_i,v_j}=0$.\\ 
 $(iv)\Rightarrow (v)$ Assume that there exist $\{u_j\}_{j\in\Lambda}\subseteq\mcl{H}_1$ and $\{v_j\}_{j\in\Lambda}\subseteq\mcl{H}_2$ as in $(iv)$. For each $k\in\Lambda$, let 
 \begin{align*}
     D_k := \{j |~ \{u_j,u_k\}\mbox{ is linearly dependent}\}
 \end{align*}
 Clearly, $k\in D_k$ for all $k\in\Lambda$. Further, given $ i,j\in\Lambda$, we have either $D_i \cap D_j = \emptyset$ or $D_i=D_j$. Let $\Lambda_0\subset \Lambda$ be such that   $D_i,D_j$ are mutually disjoint sets for every $i\neq j\in\Lambda_0$ and $\sqcup_{k\in\Lambda_0} D_k=\cup_{k\in\Lambda} D_k=\Lambda $. Thus
 \begin{align*}
    \Phi(T)=\sum_{j\in\Lambda} \ip{u_j,Tu_j}\ranko{v_j}{v_j}
           =\sum_{k\in\Lambda_0}\Big(\sum_{j\in D_k}\ip{u_j,Tu_j}\ranko{v_j}{v_j}\Big),\qquad\forall~T\in\TC{\mcl{H}_1}.
 \end{align*} 
  But, if $j\in D_k$, then $u_j=\lambda_{j,k} u_k$ for some  scalar$\lambda_{j,k}\in\mbb{C}$. Hence
 \begin{align}\label{eq-Phi-C-star}
     \Phi(T)
               =\sum_{k\in\Lambda_0}\ip{u_k,Tu_k}\Big(\sum_{j\in D_k}\abs{\lambda_{j,k}}^2\ranko{v_j}{v_j}\Big) 
               = \sum_{k\in\Lambda_0}\tr(T\wtilde{F}_k )\wtilde{R}_k,
 \end{align}
 where $\widetilde{F}_k:=\ranko{u_k}{u_k}$ and  
 \begin{align*}
    \widetilde{R}_k:=\sum_{j\in D_k}\abs{\lambda_{j,k}}^2\ranko{v_j}{v_j}
                             = \frac{1}{\norm{u_k}^4}\sum_{j\in D_k}\bip{u_j,(\ranko{u_k}{u_k})u_j}\ranko{v_j}{v_j}
 \end{align*}
 are non-zero positive operators.   Let $R_k:=\frac{\widetilde{R}_k}{\tr(\widetilde{R}_k)}\in\ST{\mcl{H}_2}$ and $F_k:=\tr(\widetilde{R}_k)\widetilde{F}_k\in\TC{\mcl{H}_1}$ for all $k\in\Lambda_0$.   Then
 \begin{align*}
       \ssum_{k\in\Lambda_0}F_k
           =\ssum_{k\in\Lambda_0}(\sum_{j\in D_k}\abs{\lambda_{j,k}}^2\norm{v_j}^2)\ranko{u_k}{u_k}
           =\ssum_{k\in\Lambda_0}\sum_{j\in D_k}\ranko{u_j}{u_j}
           =\ssum_{j\in\Lambda}\ranko{u_j}{u_j}\in\B{\mcl{H}_1}
 \end{align*}
 and $\Phi(T)=\sum_{k\in\Lambda_0}\tr(TF_k)R_k$ for all $T\in\TC{\mcl{H}_1}$. Since $\widetilde{R}_k \widetilde{R}_l=0$ we have $R_kR_l=0$  for all $k\neq l\in\Lambda_0$.\\
 $(v)\Rightarrow (vi)$ Assume that $\Phi$ has a representation as in $(v)$. By spectral theorem, write $F_j=\ranko{u_j}{u_j}$ and $R_j=\sum_{i\in\Lambda_j}\ranko{v_{i,j}}{v_{i,j}}$  for some non-zero vector $u_j\in\mcl{H}_1$ and orthogonal set $\{v_{i,j}\}_{i\in\Lambda_j}\subseteq\mcl{H}_2$ for every $j\in\Lambda$. Then, 
 \begin{align*}
     \Phi(T)=\sum_{j\in\Lambda}\tr(TF_j)R_j
        =\sum_{j\in\Lambda}\ip{u_j,Tu_j}\Big(\sum_{i\in\Lambda_j}\ranko{v_{i,j}}{v_{i,j}}\Big)
               =\sum_{j\in\Lambda}\sum_{i\in\Lambda_j}A_{i,j}TA_{i,j}^*        
      \end{align*}
 for all $T\in\TC{\mcl{H}_1}$, where $A_{i,j}:=\ranko{v_{i,j}}{u_j}$ for all $j\in\Lambda, i\in\Lambda_j$. Since $R_kR_l=0$ for $k\neq l\in\Lambda$, the set $\{v_{i,j}: i\in\Lambda_j, j\in\Lambda\}\subseteq\mcl{H}_2$ is orthogonal. Let  $\mcl{K}=\cspan\{e_{i,j}: j\in\Lambda,i\in\Lambda_j\}\subseteq\mcl{H}_2$, where $e_{i,j}=\frac{v_{i,j}}{\norm{v_{i,j}}}$, and $W\in\B{\mcl{K},\mcl{H}_2}$ be the inclusion map. Then for all $T\in\TC{\mcl{H}_1}$,
 \begin{align*}
     \mathrm{Ad}_W\circ\Phi(T)&=\sum_{j\in\Lambda}\sum_{i\in\Lambda_j}\ip{u_j, Tu_j}\ranko{W^*v_{i,j}}{W^*v_{i,j}}\\
     &=\sum_{j\in\Lambda}\sum_{i\in\Lambda_j}\norm{v_{i,j}}^2\ip{u_j,Tu_j}\ranko{e_{i,j}}{e_{i,j}}\\
     &=\sum_{j,l\in\Lambda}\sum_{i\in\Lambda_j,k\in\Lambda_l}\ip{v_{k,l},v_{i,j}}\ip{u_j,Tu_l}\ranko{e_{i,j}}{e_{k,l}}\\
     &=\sum_{j,l\in\Lambda}\sum_{i\in\Lambda_j,k\in\Lambda_l}\tr(A_{i,j}TA_{k,l}^*)\ranko{e_{i,j}}{e_{k,l}}.
 \end{align*}
 From \eqref{eq-compl-defn-2}, it follows that $\mathrm{Ad}_W\circ\Phi\in\msc{CM}(\Phi)$. Note that $\mathrm{Ad}_{WW^*}$ is the identity map on $\ran{\Phi}$  so that $\Phi=\mathrm{Ad}_{W^*}\circ(\mathrm{Ad}_W\circ\Phi)$. Since $W^*\in\B{\mcl{H}_2,\mcl{K}}$ is a co-isometry we conclude that $\Phi\in\msc{CM}(\Phi)$. \\
 $(vi)\Rightarrow (i)$ It follows trivially from the definition.
\end{proof}

\begin{rmk}\label{rmk-Deg-SEB}
 Let $\Phi:\TC{\mcl{H}_1}\to\TC{\mcl{H}_2}$ be a SEB-map.
 \begin{enumerate}[label=(\roman*)]
     \item Suppose $\Phi$ has a decomposition \eqref{eq-Holevo-SEB-degrd} for some non-zero vectors $u_j\in\mcl{H}_1$ and $v_j\in\mcl{H}_2$. With out loss of generality assume  that $v_j$'s are unit vectors. If $\Phi$ is degradable, then through a similar lines of proof of the above theorem we can see that $\ip{v_i,v_j}=0$ whenever $\{u_i,u_j\}$ is linearly independent.
     \item Suppose $\Phi(T)=\sum_j\tr(TF_j)R_j$ for some rank-one $F_j$ and states $R_j$. If $\Phi$ is degradable and  $\{F_j:j\in\Lambda\}$ is pairwise linearly independent, from $(i)$, it follows that $R_iR_j=0$ for all $i\neq j$.
 \end{enumerate} 
\end{rmk}

\begin{rmk}\label{rmk-degr-PPT}
  From \cite[Corollary 7]{MuSi22} we have
 $$\{\Phi:\M{d_1}\to\M{d_2}\mbox{ PPT and degradable maps}\}=\{\Phi:\M{d_1}\to\M{d_2}\mbox{EB and degradable maps}\}.$$
 Thus, in the finite-dimensional case Theorem \ref{thm-DEB-char} also characterizes degradable PPT-maps. 
\end{rmk}

 Let $\mcl{H}$ be a separable Hilbert space with an orthonormal basis $\{e_j\}_{j\in\Lambda}$. Define an isometry $V:\mcl{H}\to\mcl{H}\otimes\mcl{H}$ by $V(x)=\sum_j\ip{e_j,x}e_j\otimes e_j$ for all $x\in\mcl{H}$. Then the Schur product (w.r.t. $\{e_j\}_{j\in\Lambda}$) of $A,B\in\B{\mcl{H}}$ is defined as $A\odot B:=V^*(A\otimes B)V$. Note that 
 \begin{align*}
     \ip{x,V^*(A\otimes B)Vy}
        =\ip{Vx,(A\otimes B)Vy}
        =\sum_{i,j}\Bip{x,\ip{e_i,Ae_j}\ip{e_i,Be_j}\ranko{e_i}{e_j}y},
 \end{align*}
 for all $x,y\in\mcl{H}$, so that
 $$A\odot B=\ssum_{i,j}\ip{e_i,Ae_j}\ip{e_i,Be_j}\ranko{e_i}{e_j}.$$
 (Thus, if $A,B\in\B{\mcl{H}}$ has its infinite matrix representations   $[a_{ij}]_{i,j=1}^\infty$ and $[b_{ij}]_{i,j=1}^\infty$, respectively,   then $A\odot B$ has matrix representation $[a_{ij}b_{ij}]_{i,j=1}^\infty$ w.r.t $\{e_j\}_{j\in\Lambda}$.) Clearly $A\odot B\in\B{\mcl{H}}$ with $\norm{A\odot B}\leq\norm{A}\norm{B}$. Further if $A,B\in\B{\mcl{H}}^+$, then $A\otimes B\in\B{\mcl{H}\otimes\mcl{H}}^+$ so that $A\odot B\in\B{\mcl{H}}^+$ and $\tr(A\odot B)=\sum_{j}\ip{e_j,Ae_j}\ip{e_j,Be_j}$ is finite if $B\in\TC{\mcl{H}}$.  

\begin{lem}\label{lem-Schur-map}
 Let $A\in\B{\mcl{H}}^+$ and $\{e_j\}_{j\in\Lambda}$ be an orthonormal basis for $\mcl{H}$. Then  the map  $S_{A}:\TC{\mcl{H}}\to\TC{\mcl{H}}$  defined by $$S_A(T):=A\odot T \qquad \forall ~T\in\TC{\mcl{H}}$$ is a quantum operation.
\end{lem}

\begin{proof}
Let $B\in\B{\mcl{H}}$ such that $A^{\T}=B^*B$, and for $k\in\Lambda$ let   $A_k\in\B{\mcl{H}}$ be the diagonal operator given by $A_k(e_j)=\ip{e_k,Be_j}e_j$ for all $j\in\Lambda$. Then for $n\geq 1$ and $x\in\mcl{H}$,
\begin{align*}
    \bnorm{\sum_{k=1}^nA_k^*A_kx}^2
    =\sum_{j}\abs{\ip{e_j,x}}^2\Big(\sum_{k}\abs{\ip{e_k,Be_j}}^2\Big)^2
    \leq \sum_{j}\abs{\ip{e_j,x}}^2 \norm{Be_j}^4
    \leq\norm{B}^4\norm{x}^2
\end{align*}
 so that $\ssum_{k}A_k^*A_k\in\B{\mcl{H}}$. Hence  $\sum_{k}A_kTA_k^*\in\TC{\mcl{H}}$ for every $T\in\TC{\mcl{H}}$. Further (since trace norm convergence implies SOT convergence) 

\begin{align}\label{eq-Schur-map}
    \sum_{k}A_kTA_k^*=\ssum_{k}A_kTA_k^*
    &=\ssum_{k}(\ssum_{i}\ip{e_k,Be_i}\ranko{e_i}{e_i})T(\ssum_{j}\ip{Be_j,e_k}\ranko{e_j}{e_j}) \notag\\
    &=\ssum_{i,j}\sum_{k}\ip{e_k,Be_i}\ip{Be_j,e_k}\ip{e_i,Te_j}\ranko{e_i}{e_j} \notag\\
    &=\ssum_{i,j}\Bip{\sum_{k}\ip{e_k,Be_j}e_k,Be_i}\ip{e_i,Te_j}\ranko{e_i}{e_j}\notag\\
    &=\ssum_{i,j}\ip{e_j,A^{\T}e_i}\ip{e_i,Te_j}\ranko{e_i}{e_j}\notag\\
    &=\ssum_{i,j}\ip{e_i,Ae_j}\ip{e_i,Te_j}\ranko{e_i}{e_j}
    =S_A(T).
\end{align}
 Thus $S_A$ is a quantum operation.
 \end{proof}

 \begin{rmk}
  Let $A\in\B{\mcl{H}}^+$, and $S_A$ be the Schur map given in the Lemma \eqref{lem-Schur-map}. Then  any $\Psi\in\msc{CM}(S_A)$ is SEB. For, write 
  $S_A(.)=\sum_{k}A_k(.)A_k^*$ where $A_k's$ are as in the above proof. Now consider the complementary map $S_A^c$ given by the equation \eqref{eq-compl-defn-2}.i.e.,

  \begin{align}\label{eq-Schr-compl-infi}
    S_A^c(T)&=\sum_{i,j}\tr(A_iTA_j^*)\ranko{e_i}{e_j}\notag\\ 
            &=\sum_{i,j}\Big(\sum_{k}\ip{A_i^*e_k,TA_j^*e_k}\Big)\ranko{e_i}{e_j}\notag\\
            &=\sum_{i,j}\Big(\sum_{k}\Bip{\ip{Be_k,e_i}e_k,T\ip{Be_k,e_j}e_k}\Big)\ranko{e_i}{e_j}\notag\\
            &=\sum_{i,j}\sum_{k}\ip{e_i,Be_k}\ip{Be_k,e_j}\ip{e_k,Te_k}\ranko{e_i}{e_j}\notag\\
            &=\sum_{k}\ip{e_k,Te_k}\ranko{z_k}{z_k}
  \end{align}
where $z_k=\sum_{j}\ip{Be_k,e_j}e_j$. Thus $S_A^c$ is SEB. Now given any $\Psi\in\msc{CM}(S_A)$ by equation \ref{eq-compl-channel-equiv}, there exists a partial isometry $V$ such that $\Psi=\mathrm{Ad}_{V}\circ S_A^c$, hence $\Psi$ is also SEB. 
  
 \end{rmk}

\begin{thm}\label{thm-Schur-SEB-antideg}
 Let $A\in\B{\mcl{H}}^+$, and $S_A$ be the Schur map given in the lemma \eqref{lem-Schur-map}. Then the following are equivalent
    \begin{enumerate}[label=(\roman*)]
        \item Any element (and, hence all) of $\msc{CM}(S_A)$ is degradable.
        \item $A$ is a diagonal operator (i.e., $\ip{e_i,Ae_j}=0, \forall i\neq j$).
         \item $S_A$ is SEB.
    \end{enumerate}
\end{thm}

\begin{proof}
    $(i)\Rightarrow (ii)$ WLOG assume that $S_A^c$ is degradable, where $S_A^c$ given by the equation \eqref{eq-Schr-compl-infi} (If $\Psi\in\msc{CM}(S_A)$ is degradable, that is there is a quantum operation $\Gamma$ on $\TC{\mcl{H}}$ such that $\Gamma\circ S_A=\Psi$. Now by equation \ref{eq-compl-channel-equiv}, there exists a partial isometry $V$ such that $S_A^c=\mathrm{Ad}_{V}\circ \Psi$, hence $S_A^c=\mathrm{Ad}_{V}\circ\Gamma\circ S_A $ so that $S_A^c$ is degradable.). Then by Theorem \eqref{thm-DEB-char} and Remark\eqref{rmk-Deg-SEB} we get $\ip{z_i,z_j}=0$ for all $i\neq j$. But
    \begin{align*}
        \ip{z_i,z_j}
        &=\Bip{\sum_k\ip{Be_i,e_k}e_k,\sum_l\ip{Be_j,e_l}e_l}\\
        &=\sum_{k,l}\ip{e_k,Be_i}\ip{Be_j,e_l}\ip{e_k,e_l}\\
        &=\sum_k\ip{e_k,Be_i}\ip{Be_j,e_k}\\
        &=\ip{Be_j,Be_i}=\ip{e_j,A^{\T}e_i}=\ip{e_i,Ae_j}
    \end{align*}
 Thus $A$ is a diagonal operator. \\
 $(ii)\Rightarrow (iii)$ Suppose $A$ is a diagonal operator, then 
 $$A\odot T=\ssum_{i,j}\ip{e_i,Ae_j}\ip{e_i,Te_j}\ranko{e_i}{e_j}=\ssum_{j}\ip{e_j,Ae_j}\ip{e_j,Te_j}\ranko{e_j}{e_j}.$$
 Now, since the sequence $\bnorm{\sum_{j=1}^n\ip{e_j,Te_j}\ip{e_j,Ae_j}\ranko{e_j}{e_j}}_1$ converges to $\norm{A\odot T}_1$, from \cite[Theorem 2]{Gru73} we get 
\begin{align*}
    S_A(T)=\sum_{j}\ip{e_j,Te_j}\ip{e_j,Ae_j}\ranko{e_j}{e_j}
         =\sum_{j}\ip{u_j,Tu_j}\ranko{e_j}{e_j}
\end{align*}
 where $u_j=\sqrt{\ip{e_j,Ae_j}}e_j$. Hence $S_A$ is SEB.\\
$(iii)\Rightarrow (i)$ Assume that $S_A$ is SEB, hence anti-degradable. Thus any $\Psi\in\msc{CM}(S_A)$ is degradable.
\end{proof}

\section{Finite dimensional case}\label{sec-fd-case} 

 Through out this section $d,d_1,\cdots, d_4\in\mbb{N}$. We let $\{e_j^{(d)}\}_{j=1}^d\subseteq\mbb{C}^d$ denotes the standard orthonormal basis and $\mathrm{CP}(d_1,d_2)$ denotes the set of all CP-maps from $\M{d_1}$ into $\M{d_2}$. Recall that the Choi-rank of $\Phi\in\mathrm{CP}(d_1,d_2)$, denoted by $CR(\Phi)$, is the minimum number of Kraus operators required to represent $\Phi$. Given any  $\Phi\in\mathrm{CP}(d_1,d_2)$  there always exists a complementary map $\Psi\in\mathrm{CP}(d_1,r)$ with $r=CR(\Phi)\leq d_1d_2$. In fact, if $\Phi$ has a Kraus decomposition $\Phi=\sum_{i=1}^r\mathrm{Ad}_{V_i}$, then the map $\Phi^c_{\min}\in \mathrm{CP}(d_1,r)$ defined by
 \begin{align}\label{eq-compl-choirk-cp}
     \Phi_{\min}^c(X):=\sum_{i,j=1}^{r}\tr(V_i^*XV_j)\ranko{e_i^{(r)}}{e_j^{(r)}}=\sum_{k=1}^{d_2}\wtilde{V}_k^*X\wtilde{V_k},
 \end{align}
 is complementary to $\Phi$  (c.f \cite{Hol07}), where $\wtilde{V}_k=\sum_{i=1}^{d_1}\sum_{j=1}^{r}\bip{e_i^{(d_1)},V_je_k^{(d_2)}}\ranko{e_i^{(d_1)}}{e_j^{(r)}}\in\M{d_1\times r}$.
 
 If $\Psi\in\mathrm{CP}(d_1,d_3)$ is complementary to $\Phi$, then so is $\mathrm{Ad}_{V^*}\circ\Psi\in\mathrm{CP}(d_1,d_4)$ for every isometry $V\in\M{d_4\times d_3}$ and $d_4\geq d_3$. On the other hand, if $\Psi_1\in\mathrm{CP}(d_1,d_3)$ and $\Psi_2\in\mathrm{CP}(d_1,d_4)$ are two complementary maps of $\Phi$ and assume $d_3\leq d_4$, then there exists an isometry $V\in\M{d_4\times d_3}$ such that 
 \begin{align}\label{eq-compl-channel-equiv-f-dl}
    \Psi_2=\mathrm{Ad}_{V^*}\circ\Psi_1
    \qquad\mbox{and}\qquad
    \Psi_1=\mathrm{Ad}_{V}\circ\Psi_2
 \end{align}
 In fact, if $A_1:\mbb{C}^{d_1}\to \mbb{C}^{d_2}\otimes\mbb{C}^{d_3}$ and  $A_2:\mbb{C}^{d_1}\to \mbb{C}^{d_2}\otimes\mbb{C}^{d_4}$ are bounded linear such that 
 \begin{align*}
     \Phi(X)=(\id\otimes\tr)(A_jXA_j^*)
     \qquad\mbox{and}\qquad
     \Psi_j(X)=(\tr\otimes\id)(A_jXA_j^*),~~j=1,2,
 \end{align*}
 then there exists an isometry $V\in\M{d_4\times d_3}$ such that $A_2=(I_{d_2}\otimes V)A_1$
 and $A_1=(I_{d_2}\otimes V^*)A_2$. See \cite[Corollary 2.24]{Wat18}  for details.  From \eqref{eq-compl-channel-equiv-f-dl} it follows that maps that are complementary to $\Phi$ has the same Choi-rank. In fact, since   rank$(\Phi_{\min}^c(I))=r=CR(\Phi)$, from \eqref{eq-compl-channel-equiv-f-dl} it follows that $CR(\Psi)=rank(\Phi(I))\leq d_2$ for all $\Psi\in\mathrm{CP}(d_1,d_3)$ is complementary to $\Phi$. Interchanging the role of $\Phi$ and $\Psi$ we get $r\leq d_3$. Thus, from the above discussions, we have
 \begin{align*}
     \msc{CM}(\Phi)=\{\mathrm{Ad}_{V^*}\circ\Phi_{\min}^c: V\in\M{d_3\times r}, V^*V=I_r\},   
 \end{align*}
 where $\Phi_{\min}^c$ is as in \eqref{eq-compl-choirk-cp}.
 
 Let $\mathrm{UEB}(d_1,d_2)$ denotes the set of unital EB-maps from $\M{d_1}$ into $\M{d_2}$.   A linear map $\Phi\in\mathrm{UEB}(d_1,d_2)$ is said to be a \emph{$C^*$-extreme point} of $\mathrm{UEB}(d_1,d_2)$ if whenever $\Phi$ is written as  $\Phi=\sum_{i=1}^n \mathrm{Ad}_{T_i}\circ\Phi_i$, where $T_i\in\M{d_2}$ are invertible with $\sum_{i=1}^nT_i^*T_i=I_{d_2}$ and $\Phi_i\in\mathrm{UEB}(d_1,d_2)$, then  there exist unitaries $U_i\in\M{d_2}$ such that $\Phi_i=\mathrm{Ad}_{U_i}\circ\Phi$ for all $1\leq i\leq n$. In \cite{BDMS23}, Bhat et al. proved that $\Phi$ is a $C^*$-extreme point of $\mathrm{UEB}(d_1,d_2)$ if and only if $CR(\Phi)=d_2$ if and only if $\Phi=\sum_{i=1}^{d_2}\mathrm{Ad}_{\ranko{u_i}{v_i}}$ for some unit vectors $\{u_i\}_{i=1}^{d_2}\subseteq\mbb{C}^{d_1}$ and an orthonormal basis $\{v_i\}_{i=1}^{d_2}\subseteq\mbb{C}^{d_2}$.

\begin{thm}\label{thm-degr-Cstar-ext}
  Given $\Phi\in\mathrm{UEB}(d_1,d_2)$ the following are equivalent:
  \begin{enumerate}[label=(\roman*)]
      \item $\Phi$ is degradable.
      \item $\Phi$ is a $C^*$-extreme point of $\mathrm{UEB}(d_1,d_2)$.
  \end{enumerate}
\end{thm}

\begin{proof}
 $(i)\Rightarrow (ii)$ Assume that $\Phi$ is degradable. Then from  Theorem \ref{thm-DEB-char}(v),  there exist rank-one positive operators $\{F_j\}_{j=1}^n\subseteq\M{d_1}$ and states $\{R_j\}_{j=1}^n\subseteq\M{d_2}$ with $R_iR_j=0$ for all $i\neq j$ such that 
    \begin{align*}
        \Phi(X)=\sum_{j=1}^n\tr(XF_j)R_j,\qquad\forall~X\in\M{d_1}. 
    \end{align*}
 Let $F_k=\ranko{u_k}{u_k}$ and $P_k=\tr(F_k)R_k$ for all $1\leq k\leq n$. Clearly $P_kP_l=0$ for all $k\neq l$ and $\sum_{k=1}^nP_k=\Phi(I)=I$. Now let $0\neq \lambda\in\mbb{C}$ and $0\neq x\in\mbb{C}^d$ be such that $P_kx=\lambda x$. Then for all $j\neq k$ we have $P_j(x)=P_j(\frac{1}{\lambda}P_k(x))=0$. Hence $x=\sum_{j=1}^nP_j(x)=P_k(x)=\lambda x$ so that $\lambda=1$. Thus $0,1$ are only possible eigenvalues of $P_k$. Since $P_k^*=P_k$ we conclude that $P_k$'s are mutually orthogonal projections. Taking $\widetilde{u}_j=\frac{u_j}{\norm{u_j}}$ we have 
 \begin{align*}
     \Phi(X)=\sum_{j=1}^n\ip{\tilde{u}_j,X\tilde{u}_j}P_j,\qquad\forall~X\in\M{d_1},
 \end{align*}
 so that, from \cite[Theorem 5.3(v)]{BDMS23},  $\Phi$ is a $C^*$-extreme point of $\mathrm{UEB}(d_1,d_2)$.\\
 $(ii)\Rightarrow (i)$ Follows from \cite[Theorem 5.3 (vi)]{BDMS23} and Theorem \ref{thm-DEB-char}(iv).
\end{proof}

\begin{cor}\label{degr-adj-degr}
 Let $\Phi: \M{d}\to\M{d}$ be a unital  trace-preserving EB-map. Then $\Phi$ is degradable if and only if $\Phi^*$ is degradable.
\end{cor}

\begin{proof}
 Given that $\Phi,\Phi^*\in\mathrm{UEB}(d)$. Since $CR(\Phi)=CR(\Phi^*)$, from \cite[Theorem 5.3 (iv)]{BDMS23}, we have $\Phi$ is a $C^*$-extreme point if and only if $\Phi^*$ is a $C^*$-extreme point. Hence, from Theorem \ref{thm-degr-Cstar-ext}, $\Phi$ is degradable if and only if $\Phi^*$ is degradable.
\end{proof}

 
\begin{eg}
 Consider the EB-map $\Phi:\M{d}\to\M{d}$ defined by 
 $$\Phi(X)=\ip{e_1,Xe_1}I,\qquad\forall~X\in\M{d},$$
 which is unital but not trace-preserving. Note that $\Phi$ is degradable but $\Phi^*$ is not degradable. (In other words, $\Phi^*(X)=\tr(X)\ranko{e_1}{e_1}$ is a trace preserving non unital non-degradable EB-map with $(\Phi^*)^*=\Phi$ is degradable.)
\end{eg}

\begin{eg}\label{eq-degr-dual-degr-counter-eg}
 Consider the unital trace preserving EB-map  $\Phi:\M{3}\to\M{2}$ defined by 
 $$\Phi(X)=\sum_{j=1}^3\ip{e_j,Xe_j}\ranko{v_j}{v_j},$$
 where $v_1=e_1,v_2=v_3=\frac{e_2}{\sqrt{2}}\in\mbb{C}^2$. From, Theorem\ref{thm-DEB-char}, $\Phi$ is not degradable, but $\Phi^*$ is a $C^*$-extreme point, and hence degradable. 
\end{eg}
 
\begin{eg}\label{eq-Schur-degr} 
 Let $A\in\M{d}^+$ and  $S_{A}:\M{d}\to\M{d}$  be the Schur map. Since $A\in\M{d}^+$ there exist $\{v_j\}_{j=1}^{d}\subseteq\mbb{C}^{d}$ such that $A=[\ip{v_i,v_j}]$.  Write $v_j=(v_{1j},v_{2j},\cdots,v_{dj})^{\T}\in\mbb{C}^{d}$ and let $A_k=\sum_{j=1}^{d}v_{kj}\ranko{e_j}{e_j}$ for all $1\leq k\leq d$. Then $S_A=\sum_{k=1}^{d}\mathrm{Ad}_{A_k}$ so that 
 \begin{align*}
     S_A^c(X)=\sum_{k,l=1}^{d}\tr(A_k^*XA_l)\ranko{e_k}{e_l}=\sum_{j=1}^{d}\ip{e_j,Xe_j}\ranko{\ol{v_j}}{\ol{v_j}},\qquad\forall~X\in\M{d}.
 \end{align*}
 Clearly $S_A^c$ is EB.  Note that the Choi-matrix $C_{S_A^c}=\sum_{j}E_{jj}\otimes\ranko{\ol{v_j}}{\ol{v_j}}$ is a projection if and only if $v_j$'s are unit vectors if and only if diagonals of $A$ are $1$. Further,
 \begin{align*}
     (S_A^c)^*(X)=\sum_{j=1}^{d}\ip{\ol{v_j},X\ol{v_j}}\ranko{e_j}{e_j},\qquad\forall~X\in\M{d}
 \end{align*}
 is unital if and only if $v_j$'s are unit vectors. Thus, from \cite{BDMS23}, if follows that $(S_A^c)^*$ is a  $C^*$-extreme point of $\mathrm{UEB}(d)$ if and only if diagonals of $A$ are $1$  if and only if $C_{S_A^c}$ is a projection. 
\end{eg}

 Given an EB-map $\Phi:\M{d_1}\to\M{d_2}$ the minimum number of rank-one Kraus operators required to represent $\Phi$ as in \eqref{eq-Kraus-decomp-SEB} is known as the \it{entanglement breaking rank} of $\Phi$, and we denoted it as ER$(\Phi)$. From \cite{Hor97,HSR03} we have the following: 
 \begin{itemize}
     \item If $\Phi$ is unital then $d_2\leq CR(\Phi)\leq ER(\Phi)\leq (d_1d_2)^2$.
     \item If $\Phi$ is trace preserving then $d_1\leq CR(\Phi)\leq ER(\Phi)\leq (d_1d_2)^2$.
 \end{itemize} 
 The following theorem is a stronger version of \cite[Theorem 3.6]{KLPR22}. 

\begin{thm}\label{thm-PPT-proj}
 Let $\Phi:\M{d_1}\to \M{d_2}$ be a PPT-channel with rank$(\Phi(I))\leq d_1$ (this is the case if $d_2\leq d_1$).  Then the following are equivalent:
    \begin{enumerate}[label=(\roman*)]
        \item $C_{\Phi}$ is a projection.
        \item $\Phi$ is EB with  ER$(\Phi)=d_1$ (and hence $CR(\Phi)=d_1$).
        \item $\Phi^*$ is a $C^*$-extreme point of $\mathrm{UEB}(d_2,d_1)$.
        \item There exist unit vectors $\{v_j\}_{j=1}^{d_1}\subseteq\mbb{C}^{d_2}$ and an orthonormal basis $\{u_j\}_{j=1}^{d_1}\subseteq\mbb{C}^{d_1}$ such that 
        \begin{align}\label{eq-PPT-degrdble}
            \Phi(X)=\sum_{j=1}^{d_1}\ip{u_j,Xu_j}\ranko{v_j}{v_j},\qquad\forall~X\in\M{d_1}.
        \end{align}
        \item There exist $A\in\M{d_1}$ with diagonal entries equals $1$ and a unitary $U\in\M{d_1}$ such that
        \begin{align}\label{eq-Phi-Schur}
            \Phi= S_{A}^c\circ\mathrm{Ad}_{U}  
        \end{align}     
    \end{enumerate}
 Further, if $d_1=d_2$ and $\Phi$ is unital, then above conditions are equivalent to the following:
 \begin{enumerate}
     \item [(vi)] $\Phi$ is degradable. 
 \end{enumerate}
\end{thm}

\begin{proof}
 $(i)\Rightarrow (ii)$ Assume that $C_{\Phi}$ is a projection so that rank$(C_{\Phi})=\tr(C_{\Phi})$. But 
   \begin{align*}
       \tr(C_{\Phi})
       =\tr((\tr\otimes\id)C_{\Phi})
       =\tr\Big(\sum_{i,j=1}^d\tr(\ranko{e_i}{e_j})\otimes\Phi(\ranko{e_i}{e_j})\Big)
       =\tr(\Phi(I))
       =d_1.
   \end{align*}
  Since 
  \begin{align*}
       \max\{\mbox{rank}((\id\otimes\tr)(C_\Phi)),\mbox{rank}((\tr\otimes\id)(C_\Phi))\}
       =\max\{\mbox{rank}(I),\mbox{rank}(\Phi(I))\}
       =d_1,
  \end{align*}
 from \cite[Lemma 7\&8]{HSR03}, $\Phi$ is EB and  $d_1\leq$ CR$(\Phi)\leq$ ER$(\Phi)\leq d_1$.\\
 $(ii)\Rightarrow (iii)$ Assume that $\Phi$ is EB and ER$(\Phi)=d_1$. Then $\Phi^*\in\mathrm{UEB}(d_2,d_1)$ with ER$(\Phi^*)=$ ER$(\Phi)=d_1$, hence from \cite[Theorem 5.3]{BDMS23} we have $\Phi^*$ is a $C^*$-extreme point of $\mathrm{UEB}(d_2,d_1)$.\\
 $(iii)\Rightarrow (iv)$ Follows from \cite[Theorem 5.3]{BDMS23}.\\ 
 $(iv)\Rightarrow(v)$ Assume $\Phi$ is as in \eqref{eq-PPT-degrdble}. Let $U\in\M{d_1}$ be the unitary such that $U(e_i)=u_i$ for all $1\leq i\leq d_1$ and $A=[\ip{\ol{v_i},\ol{v_j}}]\in\M{d_1}$. Then, from Example \ref{eq-Schur-degr}, it follows that $\Phi=S_A^c\circ\mathrm{Ad}_{U}$. \\
 $(v)\Rightarrow (i)$ Suppose $\Phi$ is as in \eqref{eq-Phi-Schur}. Since diagonals of $A$ are $1$, we have $C_{S_A^c}$ is a projection, and hence $C_\Phi=(U^{\T}\otimes I)^*C_{S_A^c}(U^{\T}\otimes I)$ is also a projection. \\
 $(iii)\Leftrightarrow (vi)$ Let $\Phi:\M{d}\to\M{d}$ be a unital PPT-channel. Assume that $\Phi^*\in\mathrm{UEB}(d)$ is a $C^*$-extreme point. Then from the Theorem \eqref{thm-degr-Cstar-ext} and Corollary \eqref{degr-adj-degr}, $\Phi$ is degradable. Conversely assume that $\Phi$ is degradable. Then from  \cite[Corollary 7]{MuSi22} we get $\Phi$ is EB, and hence from Theorem \eqref{thm-degr-Cstar-ext} and Corollary \eqref{degr-adj-degr} it follows that $\Phi^*$ is a $C^*$-extreme point of $\mathrm{UEB}(d)$.
\end{proof}
 
 In the above Theorem, one may replace \eqref{eq-Phi-Schur} by  $\Phi=\mathrm{Ad}_{V^*}\circ (S_A)_{\min}^c\circ\mathrm{Ad}_U$ where $U\in\M{d_1}$ is a unitary, $V\in\M{d_3\times r}$ is an isometry and $r=CR(S_A)$. Example \ref{eq-degr-dual-degr-counter-eg} shows that, in Theorem \ref{thm-PPT-proj}, $(vi)$ is not equivalent to any of the other statements when $d_1\neq d_2$. 
 
  Note that  that if $\Phi\in\mathrm{CP}(d_1,d_2)$ is unital, then the CP-map $\Phi^c_{\min}$ need not be unital. From \cite[Lemma 2.4]{KLPR22} we have $\Phi^c_{\min}$ is unital if and only if  $C_{\Phi}$ is a projection. So we ask what are all CP-maps $\Phi$ for which $C_\Phi$ is a projection. If $d_2\leq d_1$, then Theorem \ref{thm-PPT-proj} says that
\begin{align*}
    \{\Phi:\M{d_1}\to\M{d_2}\mbox{ PPT with }C_\Phi\mbox{ is projection}\}
    =\{\Phi:\M{d_1}\to\M{d_2}\mbox{ EB with }C_\Phi\mbox{ is projection}\},
\end{align*}
 and further characterize the elements of the first set.

\begin{rmk}
 Let $\Phi:\M{d_1}\to\M{d_2}$ be a PPT-channel with rank$(\Phi(I))\leq d_1$. If ${\Phi}$ is degradable, then $(\Phi^c)^*$ is a $C^*$-extreme point of $\mathrm{UEB}(d_2,d_1)$. For, since ${\Phi}$ is degradable, \cite[Corollary 7]{MuSi22} we have both $\Phi$ and $\Phi^c$ are EB.  As 
 $$d_1\leq CR((\Phi^c)^*)=CR(\Phi^c)=rank(\Phi(I))\leq d_1,$$  from \cite[Theorem 5.3]{BDMS23} we get $(\Phi^c)^*$ is a $C^*$-extreme point.
\end{rmk}

\begin{eg}
 Given $\lambda\in [-1, 1/d]$ and $d>1$ consider the map $\mcl{W}_\lambda:\M{d}\to\M{d}$ given by 
 $$\mcl{W}_\lambda(X)=\frac{1}{d-\lambda}(\tr(X)I-\lambda X^{\T}),$$
 which is an unital trace preserving EB-map (\cite{WeHo02}). Consider the Choi-matrix
 $$C_{\mcl{W}_\lambda}=\frac{1}{d-\lambda}\big((I\otimes I)-\lambda F\big),$$
 where $F\in\B{\mbb{C}^d\otimes\mbb{C}^d}$ is the flip operator, and satisfies $F^2=I_d\otimes I_d$. Since $C_{\mcl{W}_{\lambda}}$ is not a projection $\mcl{W}_{\lambda}$ is not degradable.   
\end{eg}

\begin{eg}
 Given $\lambda\in[-\frac{1}{d+1},1]$ the map $\Phi_{\lambda,d}:\M{d}\to\M{d}$ defined by   
  \begin{align}\label{eq-Phi-lambda}
      \Phi_{\lambda,d}(X):=\frac{1}{2\lambda +d}\{\tr(X)I+ \lambda(X+X^{\T})\}
  \end{align}
  is an EB-map (\cite{DMS23}). Note that
 \begin{align*}
       C_{\Phi_{\lambda,d}} =(I_d\otimes I_d)+\lambda(P+F),
  \end{align*}
  where $P=\ranko{\sum_j e_j\otimes e_i}{\sum_j e_j\otimes e_j}\in\B{\mbb{C}^d\otimes\mbb{C}^d}$. Since  $C_{\Phi_{\lambda,d}}$ is not a projection, $\Phi_{\lambda,d}$ is not degradable.   
\end{eg}
 
 In the rest of this section we discuss degradable CP-maps.  Suppose $\Phi\in\mathrm{CP}(d,d)$ is of the form $\Phi=\mathrm{Ad}_{Y}\circ S_A\circ\mathrm{Ad}_Z$, where $A\in\M{d}^+, Y,Z\in\M{d}$. If both $Y,Z$ are invertible (in particular unitaries), then $\Phi$ is degradable. See \cite{WoPe07,DeSh05} for details.
 
\begin{eg}
 Let $W\in\M{d_3\times d_2}$ is a co-isometry, $A\in\M{d_3}^+$ and $Z\in\M{d_1\times d_3}$. Consider $\Phi=\mathrm{Ad}_W\circ S_A\circ \mathrm{Ad}_Z\in\mathrm{CP}(d_1,d_2)$.  Now since $S_A^c$ is EB we have $\Phi^c=(\mathrm{Ad}_{W}\circ S_A\circ\mathrm{Ad}_{V})^c=S_A^c\circ \mathrm{Ad}_{V}$ is also EB, hence anti-degradable. Thus $\Phi$ is degradable.
\end{eg} 

 \begin{prop}
 Let $\Phi_j:\M{d}\to\M{d_j},1\leq j\leq k$ be pure CP-maps. Then the map $\Phi:=\oplus_{i=1}^{k}\Phi_i$ is degradable.  
 \end{prop}
 
 \begin{proof}
  We prove it only for the case $k=2$, and the general case can be proved similarly. So, assume $\Phi_j=\mathrm{Ad}_{V_j}$ for $j=1,2$.  Letting $W_1=\Matrix{V_1&0}\in\M{d\times{d_1+d_2}}$ and $W_2=\Matrix{0&V_2}\in\M{d\times{d_1+d_2}}$ we have $\Phi=\sum_{j=1}^2\mathrm{Ad}_{W_j}$. Note that $W_jW_i^*=0$ for $i\neq j\in\{1,2\}$ and hence 
   \begin{align*}
       \Phi^c(X) =\sum_{i,j=1}^2\tr(W_i^*XW_j)\ranko{e_i^{(2)}}{e_j^{(2)}}
                       =\sum_{j=1}^2\tr(W_j^*XW_j)\ranko{e_j^{(2)}}{e_j^{(2)}}
                       =\sum_{j=1}^2\tr(V_j^*XV_j)\ranko{e_j^{(2)}}{e_j^{(2)}}
   \end{align*}  
   for all $X\in\M{d}$. Note that  $\Gamma:\M{d_1+d_2}\to\mbb{C}$ given by 
   $$\Gamma(\Matrix{A_{11}&A_{12}\\A_{21}&A_{22}}):=\Matrix{\tr(A_{11})&0\\0&\tr(A_{22})}$$
   is a CP-map such that  $\Phi^c=\Gamma\circ\Phi$ so that $\Phi$ is degradable. 
 \end{proof}

\begin{cor}
 Let $\Phi\in\mathrm{UCP}(d_1,d_2)$.  If $\Phi$ is  a $C^*$-extreme point of $\mathrm{UCP}(d_1,d_2)$, then it is degradable.
\end{cor} 

 \begin{proof}
 Let $\Phi\in \mathrm{UCP}(d_1,d_2)$ be a $C^*$-extreme point. From \cite[Theorem 2.1]{FaZh98}, 
 there exist a unitary $U\in\M{d_2}$ and positive integers $n_1\geq n_2 \geq \cdots\geq n_k$ with $\sum_{i=1}^k n_i=d_2$ and  isometries $V_i\in\M{d_1\times n_i}$  such that   $\mathrm{Ad_U}\circ\Phi=\oplus_{i=1}^k \mathrm{Ad}_{V_i}$. From the above proposition it follows that $\Phi^c=(\mathrm{Ad_U}\circ\Phi)^c=\Gamma\circ\mathrm{Ad}_U\circ\Phi$ for some CP-map $\Gamma$ concluding that $\Phi$ is degradable. 
\end{proof} 

\begin{eg}\label{eg-counter-CP-extr-degr}
 Let $\Phi:\M{d}\to\M{d}$ be a unital EB map given by $\Phi=\sum_{i=1}^d\mathrm{Ad}_{A_i}$ where $A_i=\ranko{e_i}{e_i}$.  Since  $A_i^*A_j=0$ for all $i\neq j$, we have $\Phi=\Phi^c$. Thus From theorem \ref{thm-DEB-char}, $\Phi$ is degradable. Note that $\Phi$ is not a linear extreme (hence not $C^*$-extreme) point of $\mathrm{UCP}(d)$.   
\end{eg}

\section{Acknowledgement}
RD supported by the IoE-CoE project (No. SB20210797MAMHRD008573) from  MHRD (Ministry of Human Resource Development, India). GS thanks to IIT Madras for financial support through the Institute Postdoctoral Fellowship. KS is partially supported by the IoE-CoE Project (No. SB22231267MAETWO008573) from MHRD (India) and partially by the MATRIX grant (File no. MTR/2020/000584) from SERB (India).

\appendix
\section{}

The following theorem is a generalized version of \cite[Theorem 2]{Hol07} to infinite dimensional Hilbert spaces and we use the similar technique to prove this. 

\begin{thm}\label{thm-Stine-QC-uniq}
 Let $\Phi:\TC{\mcl{H}_1}\to\TC{\mcl{H}_2}$ be a quantum operation with the minimal Stinespring representation  $(\widehat{\mcl{K}},\widehat{A})$.  If $(\mcl{K},A)$ is a Stinespring representation of $\Phi$, then  there exists an isometry $V:\widehat{\mcl{K}}\to\mcl{K}$ such that 
       \begin{align*}
                 A=(I_{\mcl{H}_2}\otimes V)\widehat{A}
                  \quad\mbox{and}\quad
                 \widehat{A}=(I_{\mcl{H}_2}\otimes V^*)A. 
       \end{align*} 
 (If $(\mcl{K},A)$ is also minimal then $V$ can be chosen to be a unitary).  Consequently, given other Stinespring representation $(\mcl{K}',A')$ of $\Phi$  there exist a partial isometry $W:\mcl{K}\to\mcl{K}'$ such that 
          \begin{align*}
                  A'=(I_{\mcl{H}_2}\otimes W)A   
                  \qquad\mbox{and}\qquad
                  A=(I_{\mcl{H}_2}\otimes W^*)A' 
          \end{align*} 
         
\end{thm} 

 \begin{proof}
 Let $(\widehat{\mcl{K}}, \widehat{A})$  be the  minimal Stinespring representation of $\Phi^*$. Then,
 \begin{align*}
     \cspan\{(X\otimes I_{\widehat{\mcl{K}}})\widehat{A}h: h\in\mcl{H}_1, X\in\B{\mcl{H}_2}\}=\mcl{H}_2\otimes\widehat{\mcl{K}}
 \end{align*}
 and $\Phi^*(X)=\widehat{A}^*(X\otimes I_{\widehat{\mcl{K}}})\widehat{A}$ for all $X\in\B{\mcl{H}_2}.$ Now since $(\mcl{K}, A)$ is a Stinespring representation for $\Phi$ we have $\Phi^*(X)=A^*(X\otimes I_{\mcl{K}})A$ for all $X\in\B{\mcl{H}_2}$. Define $S:\mcl{H}_2\otimes \widehat{\mcl{K}}\to \mcl{H}_2\otimes \mcl{K}$ by 
 \begin{align*}
    S(X\otimes I_{\widehat{\mcl{K}}})\widehat{A}h:=(X\otimes I_{\mcl{K}})Ah,\qquad\forall~X\in\B{\mcl{H}_2},h\in\mcl{H}_1.
 \end{align*}
 Then $S$ is an isometry  satisfying $S\widehat{A}=A$ and $S(X\otimes I_{\widehat{\mcl{K}}})=(X\otimes I_{\mcl{K}})S$ for all $X\in\B{\mcl{H}_2}$. Hence there exists an isometry  $V\in\B{\widehat{\mcl{K}},\mcl{K}}$ such that $S=I_{\mcl{H}_2}\otimes V$. Now $S\widehat{A}=A$ implies that $A=(I_{\mcl{H}_2}\otimes V)\hat{A}$ and hence $\hat{A}=(I_{\mcl{H}_2}\otimes V^*)A$.  Similarly, one can have an isometry $V':\hat{\mcl{K}}\to\mcl{K}'$  such that $A'=(I_{\mcl{H}_2}\otimes V')\hat{A}$ and $\hat{A}=(I_{\mcl{H}_2}\otimes V'^*)A'$. Note that $W:=V'V^*\in\B{\mcl{K},\mcl{K}'}$ is the required partial isometry.
\end{proof}

\begin{prop}\label{Comp_channel_formula}
Let $\Phi:\TC{\mcl{H}_1}\to\TC{\mcl{H}_2}$ be a quantum operation with Kraus decomposition $\Phi(T)=\sum_{j\in\Lambda}A_jTA_j^*$, where $\{A_j\}_{j\in\Lambda}\subseteq\B{\mcl{H}_1,\mcl{H}_2}$ is countable family with $\ssum_{j\in\Lambda}A_j^*A_j\in\B{\mcl{H}_2}$. If $\mcl{H}_3$ is any other separable complex Hilbert space with orthonormal basis $\{e_j:j\in \Lambda\}$, then the map $\Psi:\TC{\mcl{H}_1}\to\TC{\mcl{H}_3}$ given by \begin{align}\label{comp of chan}
     \Psi(T):=\sum_{i,j\in \Lambda}\tr(A_iTA_j^*)\ranko{e_i}{e_j}
 \end{align}
 is a well-defined quantum operation and complementary to $\Phi$.
\end{prop}

\begin{proof}
 Since $\ssum_jA_j^*A_j\in\B{\mcl{H}_2}^+$, we have
 \begin{align*}
     \bnorm{\sum_{j\in\Lambda} A_jx\otimes e_j}^2
     =\ip{x,\sum_{j\in\Lambda}A_j^*A_jx}
     \leq \bnorm{\ssum_{j\in\Lambda}A_j^*A_j}\norm{x}^2,\qquad\forall~x\in\mcl{H}_1.
 \end{align*} 
 Hence $x\mapsto\sum_{j\in\Lambda}A_jx\otimes e_j$ defines an operator $A\in\B{\mcl{H}_1,\mcl{H}_2\otimes\mcl{H}_3}$. Now fix an orthonormal basis $\{f_j\}_{j\in\Lambda'}$ for $\mcl{H}_2$. Then for all $T\in\TC{\mcl{H}_1}$, we have
 \begin{align*}
    \tr_{\mcl{H}_3}(ATA^*)
       &=\sum_{i,j\in\Lambda'}\sum_{k\in\Lambda}\ip{f_i\otimes e_k,ATA^*(f_j\otimes e_k)}\ranko{f_i}{f_j}\\
       &=\sum_{i,j\in\Lambda'}\ip{f_i,\sum_kA_kTA_k^*f_j}\ranko{f_i}{f_j}\\
       &=\sum_kA_kTA_k^*\\
       &=\Phi(T)
 \end{align*}
 and 
 \begin{align*}
    \tr_{\mcl{H}_2}(ATA^*)
       &=\sum_{k\in\Lambda'}\sum_{i,j\in\Lambda}\ip{f_k\otimes e_i,ATA^*(f_k\otimes e_j)}\ranko{e_i}{e_j}\\
       &=\sum_{i,j\in\Lambda}\sum_{k\in\Lambda'}\ip{f_k,A_iTA_j^*f_k}\ranko{e_i}{e_j}\\
       &=\sum_{i,j\in\Lambda}\tr(A_iTA_j^*)\ranko{e_i}{e_j}.
 \end{align*}
 Thus $\Psi(T):=\tr_{\mcl{H}_{2}}(ATA^*)= \sum_{i,j\in\Lambda}\tr(A_iTA_j^*)\ranko{e_i}{e_j} ,\forall~T\in\TC{\mcl{H}_1}$ is a well-defined quantum operation which is complementary to $\Phi$. 
\end{proof}

 We used similar techniques of \cite{He13,LiDu15} to prove the following theorem.

\begin{thm}\label{thm-QO-char}
 Let $\mcl{H}_1,\mcl{H}_2$ be two separable Hilbert spaces with $\dim(\mcl{H}_1)=\infty$. Then for a quantum operation $\Phi:\TC{\mcl{H}_1} \to \TC{\mcl{H}_2}$ the following are equivalent: 
 \begin{enumerate}[label=(\roman*)]
    \item $\Phi$ is SEB.
     \item For every sequence of positive scalars $\{\lambda_i:i\in\mbb{N}\}\subseteq\mbb{R}$ with $\sum_{i\in\mbb{N}}\lambda_i<\infty$,  the operator $\sum_{i,j\in\mbb{N}}\sqrt{\lambda_i\lambda_j}E_{ij}\otimes \Phi(E_{ij})\in\TC{\mcl{H}_1\otimes\mcl{H}_2}^+$ is countably decomposable separable.
    \item For one sequence of positive scalars $\{\lambda_i:i\in\mbb{N}\}\subseteq\mbb{R}$ with $\sum_{i\in\mbb{N}}\lambda_i<\infty$,  the operator $\sum_{i,j\in\mbb{N}}\sqrt{\lambda_i\lambda_j}E_{ij}\otimes \Phi(E_{ij})\in\TC{\mcl{H}_1\otimes\mcl{H}_2}^+$ is countably decomposable separable.
    \item (Kraus decomposition:) There exist Kraus operators $\{A_j\}_{j\in \mbb{N}}\subseteq\B{\mcl{H}_1,\mcl{H}_2}$ of rank at most one with $\ssum_{j\in\mbb{N}}A_j^*A_j\in\B{\mcl{H}_1}$ such that 
    \begin{align*}
            \Phi(T)=\sum_{j\in\mbb{N}}A_jTA_j^*,\qquad\forall~T\in\TC{\mcl{H}_1}.
    \end{align*}
     \item (Holevo form:) There exist positive operators $\{F_j\}_{j\in\mbb{N}}\subseteq\B{\mcl{H}_1}$ with $\ssum_{k\in\mbb{N}}F_k\in\B{\mcl{H}_1}$ and  states $\{R_j\}_{j\in\mbb{N}}\subseteq\ST{\mcl{H}_2}$ such that 
    \begin{align*}
        \Phi(T)= \sum_{j\in\mbb{N}}\tr(TF_j)R_j,\qquad\forall~T\in\TC{\mcl{H}_1}.
    \end{align*}
   
    \end{enumerate}
\end{thm}

\begin{proof}
    $(i)\Rightarrow (ii)$ Let $\{e_j\}_{j\in\mbb{N}}$ be an orthonormal basis of $\mcl{H}_1$ and let $\lambda_j\in\mbb{R}$ be positive scalars such that $\sum_{j\in\mbb{N}}\lambda_j<\infty$. Let $e:=\sum_{j\in\mbb{N}}\sqrt{\lambda_j}e_j\otimes e_j\in\mcl{H}_1\otimes\mcl{H}_1$ and $\rho:=\ranko{e}{e}\in\TC{\mcl{H}_1\otimes\mcl{H}_1}$. Let $w_n:=\sum_{j=1}^n\sqrt{\lambda_j}e_j\otimes e_j$ and $\rho_n:=\ranko{w_n}{w_n}$.  Clearly, $\lim_n\tr(\rho_n)=\tr(\rho)$ and $\rho_n\to \rho$ in WOT, hence $\rho_n\to \rho$ in trace-norm (\cite[Theorem 2]{Gru73}). Note that $\sum_{i,j}\norm{\sqrt{\lambda_i\lambda_j}E_{ij}\otimes E_{ij}}_1$ is convergent, hence (w.r.t trace norm topology)
 \begin{align*}
     \lim_n\rho_n
     &=\lim_{n\to\infty}\sum_{i,j=1}^n\sqrt{\lambda_i\lambda_j}E_{ij}\otimes E_{ij}\\ &=\lim_{n,m\to\infty}\sum_{i=1}^n\sum_{j=1}^m\sqrt{\lambda_i\lambda_j}E_{ij}\otimes E_{ij}\\
     &=\sum_{i,j}\sqrt{\lambda_i\lambda_j}E_{ij}\otimes E_{ij}
 \end{align*}
 Since $\id_{\mcl{H}_1}\otimes\Phi$ is continuous in trace norm we have
 \begin{align}\label{eq-id-Phi-rho-1}
     (\id_{\mcl{H}_1}\otimes\Phi)(\rho)
     =\sum_{i,j}\sqrt{\lambda_i\lambda_j}E_{ij}\otimes\Phi(E_{ij}).
 \end{align}
 As $\Phi$ is SEB, we have $(\id_{\mcl{H}_1}\otimes\Phi)(\rho)$  is  countably decomposable separable.\\
 $(ii)\Rightarrow (iii)$ Clearly holds.\\
 $(iii)\Rightarrow (iv)$ Let $\{\lambda_i:i\in\mbb{N}\}\subseteq\mbb{R}$ with $\sum_{i\in\mbb{N}}\lambda_i<\infty$ be such that  $\sum_{i,j\in\mbb{N}}\sqrt{\lambda_i\lambda_j}E_{ij}\otimes \Phi(E_{ij})\in\TC{\mcl{H}_1\otimes\mcl{H}_2}^+$ is countably decomposable separable. Then there exist unit vectors $u_k, v_k\in\mcl{H}_1$ and non-negative scalars $r_k$ such that
 \begin{align}\label{eq-id-Phi-rho-2}
    \sum_{i,j\in\mbb{N}}\sqrt{\lambda_i\lambda_j}E_{ij}\otimes \Phi(E_{ij})
    =\sum_{k\in\mbb{N}}r_k\ranko{u_k}{u_k}\otimes \ranko{v_k}{v_k}.
 \end{align}
 Then for all $x,y\in\mcl{H}_2$
 \begin{align*}
    \sqrt{\lambda_i\lambda_j}\ip{y,\Phi(E_{ij})x}
    &=\ip{e_i\otimes y,(\sum_{s,t}\sqrt{\lambda_s\lambda_t}E_{st}\otimes \Phi(E_{st}))(e_j\otimes x)}\\
    &=\ip{e_i\otimes y, (\sum_{k}r_k\ranko{u_k}{u_k}\otimes \ranko{v_k}{v_k}) (e_j\otimes x)}\\
    &=\sum_{k}r_k\ip{u_k,e_j}\ip{e_i,u_k}\ip{v_k,x}\ip{y,v_k}\\
    &=\ip{y,\Big(\sum_k r_k\ip{u_k,e_j}\ip{e_i,u_k}\ranko{v_k}{v_k}\Big)x}
 \end{align*}
 Therefore,
 \begin{align}\label{eq-Phi(Eij)}
     \Phi(E_{ij})=\frac{1}{\sqrt{\lambda_i\lambda_j}}\sum_kr_k\ip{u_k,e_j}\ip{e_i,u_k}\ranko{v_k}{v_k},\quad\forall~i,j.
 \end{align}
 Let $v=\sum_j\sqrt{\lambda_j}e_j\in\mcl{H}_1$ and $A=\ranko{v}{v}=\sum_{i,j}\sqrt{\lambda_i\lambda_j}E_{ij}\in\TC{\mcl{H}_1}^+$. Applying partial trace to \eqref{eq-id-Phi-rho-1} and \eqref{eq-id-Phi-rho-2} we get 
 \begin{align}\label{partial trace bounded}
   \sum_{k}r_k\ranko{u_k}{u_k}
   &=\sum_{i,j}\sqrt{\lambda_i\lambda_j}E_{ij}\tr\Big( \Phi(E_{ij})\Big) \notag\\
   &=\sum_{i,j}\sqrt{\lambda_i\lambda_j}E_{ij}\tr\Big( E_{ij}\Phi^*(I)\Big) \notag\\
   &=\sum_{ij}\sqrt{\lambda_i\lambda_j}\ip{e_j,\Phi^*(I)e_i}E_{ij} \notag\\ 
   &=A\odot\Phi^*(I)^{\T} \notag\\
   &\leq A\odot\norm{\Phi^*(I)^{\T}} I \notag\\
   &=\norm{\Phi^*(I)}\sum_{j=1}^\infty\lambda_jE_{jj}. 
 \end{align}
 Let $P_n=\sum_{j=1}^n\ranko{e_j}{e_j}$ and $z_{k,n}:=\sum_{i=1}^n\frac{\sqrt{\mu_k}}{\sqrt{\lambda_i}}\ip{u_k,e_i}e_i$ for $k,n\geq 1$.  From \eqref{partial trace bounded}, we get $ \sum_{k=1}^\infty r_k \ranko{P_nu_k}{P_nu_k}\leq\norm{\Phi^*(I)}\sum_{i=1}^n\lambda_iE_{ii}$ so that 
 \begin{align}\label{eq-Fk-tilde}
   \sum_{k=1}^m\ranko{z_{k,n}}{z_{k,n}}
    &=\sum_{k=1}^m\sum_{i,j=1}^n\frac{r_k}{\sqrt{\lambda_i\lambda_j}}\ip{u_k,e_i}\ip{e_j,u_k}E_{ij} \notag \\
    &=\sum_{k=1}^m\sum_{i,j=1}^n\frac{r_k}{\sqrt{\lambda_i\lambda_j}}\ip{e_j,\ranko{u_k}{u_k}e_i}E_{ij} \notag \\
    &=\sum_{i,j=1}^n\sum_{k=1}^m\frac{r_k}{\sqrt{\lambda_i\lambda_j}}\ip{P_ne_j,\ranko{u_k}{u_k}P_ne_i}E_{ij} \notag\\
    &=\sum_{i,j=1}^n\frac{1}{\sqrt{\lambda_i\lambda_j}}\ip{e_j,\sum_{k=1}^mr_k\ranko{P_nu_k}{P_nu_k}e_i}E_{ij}\notag\\
    &\leq\sum_{i,j=1}^n\frac{1}{\sqrt{\lambda_i\lambda_j}}\ip{e_j,\sum_{k=1}^\infty r_k\ranko{P_nu_k}{P_nu_k}e_i}E_{ij}\notag\\
    &\leq \norm{\Phi^*(I)} P_n,\qquad\forall~m,n\geq 1.
 \end{align}
 Since $\ranko{z_{k,n}}{z_{k,n}}\leq \sum_{j=1}^{k+1}\ranko{z_{j,n}}{z_{j,n}}\leq \norm{\Phi^*(I)} P_n$ for all $k\geq 1$ we have
 \begin{align*}
    \norm{z_{k,n}}^2&=\bnorm{\ranko{z_{k,n}}{z_{k,n}}} \leq \norm{\Phi^*(I)},\qquad\forall~n,k\geq 1.
\end{align*}
 Thus $z_{k,n}$ converges, say to  $z_k=\sum_{i=1}^{\infty}\frac{\sqrt{\mu_k}}{\sqrt{\lambda_i}}\ip{u_k,e_i}e_i\in\mcl{H}_1$. Now for every $m,n\geq 1$, from \eqref{eq-Fk-tilde}, we have  
 $$\sum_{k=1}^m\ranko{z_{k,n}}{z_{k,n}}=\sum_{k=1}^m P_n\ranko{z_k}{z_k}P_n\leq \norm{\Phi^*(I)} P_n,$$
 i.e., $ P_n(\norm{\Phi^*(I)} I-\sum_{k=1}^m\ranko{z_k}{z_k})P_n$ is positive. From \cite[Lemma 3.3]{LiDu15}, it follows that $\norm{\Phi^*(I)} I-\sum_{k=1}^m\ranko{z_k}{z_k}$ is positive for every $m\geq 1$. By letting $A_k:=\ranko{v_k}{v_k}\in\B{\mcl{H}_1,\mcl{H}_2}$ we get 
 $\ssum_k A_k^*A_k=\ssum_k\ranko{z_k}{z_k}\leq \norm{\Phi^*(I)} I_{\mcl{H}_1}$. Consider the quantum operation $\Psi:\TC{\mcl{H}_1}\to\TC{\mcl{H}_2}$ defined  by $\Psi(\cdot)=\sum_{k}A_k(\cdot)A_k^*$. Then for every $i,j\geq 1$
 \begin{align*}
     \Psi(E_{ij})
                 =\sum_{k}\ip{z_k,e_i}\ip{e_j,z_k}\ranko{v_k}{v_k}
                 =\sum_{k}\frac{r_k}{\sqrt{\lambda_i\lambda_j}}\ip{e_i,u_k}\ip{u_k,e_j}\ranko{v_k}{v_k}
                 =\Phi(E_{ij})
 \end{align*}
 Now if $T=\sum_{i,j}t_{ij}E_{ij}\in\TC{\mcl{H}_1}$ for some $t_{ij}\in\mbb{C}$, then since $\Phi,\Psi$ are continuous in trace norm we get
 \begin{align*}
     \Phi(T)=\lim_{n\to\infty}\sum_{i,j=1}^nt_{ij}\Phi(E_{ij})
            =\lim_{n\to\infty}\sum_{i,j=1}^nt_{ij}\Psi(E_{ij})
            =\Psi(T).
 \end{align*}
 Thus $\Phi(T)=\sum_{k}A_kTA_k^*$ for all $T\in\TC{\mcl{H}_1}.$\\
 $(iv)\Rightarrow (v)$ Assume that  $\Phi$ has a represenation as in \eqref{eq-Kraus-decomp-SEB}. Let $\Lambda=\{j: A_j\neq 0\}$ and  write $A_j=\ranko{v_j}{u_j}$ for some $u_j\in\mcl{H}_1$ and unit vector $v_j\in\mcl{H}_2$ . For every $j\in\Lambda$, let  $F_j:=A_j^*A_j$ and $R_j:=\ranko{v_j}{v_j}$, then $\ssum_{j\in\Lambda}F_j\in\B{\mcl{H}_1}$ and
 \begin{align*}
     \Phi(T)&=\sum_{j\in\Lambda}A_jTA_j^*\\
            &=\sum_{j\in\Lambda}\ranko{v_j}{u_j}T\ranko{u_j}{v_j}\\
            &=\sum_{j\in\Lambda}\tr(T\ranko{u_j}{u_j})\ranko{v_j}{v_j}\\
            &=\sum_{j\in\Lambda}\tr(TF_j)R_j\\
            &=\sum_{j\in\mbb{N}}\tr(TF_j)R_j,
 \end{align*}
 where $F_j=0$ and $R_j=\ranko{v}{v}$ for all $j\notin\Lambda$ and for some fixed unit vector $v\in\mcl{H}_2$.\\
 $(v)\Rightarrow (i)$ Assume that $\Phi$ has a Holevo form as in \eqref{eq-Holevo-SEB}. Now let $\mcl{K}$ be a separable Hilbert space and let $\rho\in\TC{\mcl{K}\otimes \mcl{H}_1}^+.$ Fixing an orthonormal basis $\{e_j:j\in\Lambda\}$ of $\mcl{K}$ write $\rho=\sum_{i,j\in\Lambda}E_{ij}\otimes \T_{ij}$ for some $T_{ij}\in\TC{\mcl{H}_1}$, where $E_{ij}=\ranko{e_i}{e_j}$. Then
 \begin{align*}
     \id_{\mcl{K}}\otimes\Phi(\rho)&=\sum_{ij}E_{ij}\otimes \Phi(T_{ij})\\
                                   &=\sum_{ij}E_{ij}\otimes \Big(\sum_{k}\tr(T_{ij}F_k)R_k\Big)\\
                                   &=\sum_{ij}\Big(\sum_{k}\tr(T_{ij}F_k)E_{ij}\Big)\otimes R_k\\
                                   &=\sum_{k}\Big(\sum_{ij}\tr(T_{ij}F_k)E_{ij}\Big)\otimes R_k\\
                                   &=\sum_{k}\tr_{\mcl{H}_2}\Big((I_{\mcl{K}}\otimes F_k)\rho\Big)\otimes R_k
 \end{align*}
 Since $\tr_{\mcl{H}_2}\Big((I_{\mcl{K}}\otimes F_k)\rho\Big)$ is positive $(\id_{\mcl{K}}\otimes\Phi)\rho$ is countably decomposable separable. Since $\mcl{K}$ and $\rho$ are arbitrary $\Phi$ is SEB.
 \end{proof}

 \bibliographystyle{alpha}

\end{document}